\documentclass[14pt]{amsart}
\usepackage{amsfonts}
\usepackage{amsfonts,latexsym,rawfonts,amsmath,amssymb,amsthm}
\usepackage[plainpages=false]{hyperref}

\usepackage{graphicx}
\RequirePackage{color}

\numberwithin{equation}{section}

\def\p{\partial }

\def\L{\mathcal L}
\def\M{\mathcal M}
\def\F{\mathcal F}

\def\R{\Bbb R}

\def\Om{\Omega}
\def\pom{\p  \Omega}

\def\I{\mathcal I}
\def\D{\mathcal D}
\def\A{\mathcal A}
\def\B{\mathcal B}
\def\P{\mathcal P}

\def\G{\mathcal G}
\def\T{\mathcal T}

\newtheorem{Proposition}{Proposition}[section]
\newtheorem{Theorem}[Proposition]{Theorem}
\newtheorem{Lemma}[Proposition]{Lemma}

\newtheorem{Remark}[Proposition]{Remark}
\newtheorem{Example}[Proposition]{Example}
\newtheorem{Definition}[Proposition]{Definition}

\title  {Boundary $C^{2, \alpha}$ Regularity for the Oblique  Boundary Value Problem of  Monge-Amp\`ere Equations   }\thanks{This work was supported by NSFC   12141103.}
\begin{document}

\address{Huaiyu Jian: Department of Mathematical Sciences, Tsinghua University, Beijing 100084, China.}

 \address{Xushan Tu: Department of Mathematics,  The Hong Kong University of Science and Technology, Clear Water Bay, Kowloon, Hong Kong.}



\email{hjian@tsinghua.edu.cn; \ \  maxstu@ust.hk  }


\bibliographystyle{plain}

%
\maketitle

\baselineskip=15.8pt
\parskip=3pt

\centerline {\bf  Huaiyu Jian}
\centerline {Department of Mathematical sciences, Tsinghua University}
\centerline {Beijing 100084, China}

\vskip10pt

 \centerline { \bf Xushan Tu}
\centerline {Department of Mathematics, The Hong Kong University of Science and Technology }
\centerline {Clear Water Bay, Kowloon, Hong Kong}

  \vskip 15pt

\begin{abstract}
   We study the good shape property of boundary sections of convex solutions of the oblique boundary value problem for
     Monge-Amp\`ere equations
     \[	\det  D^2u  =f(x)      \text{ in } \Om  , \quad D_{\beta}u  = \phi(x)                 \text{ on } \pom.\]
       In the two-dimensional case, we prove the global $C^{2,\alpha}$   estimate
     for the solution.  When  the dimension $n \geq 3$, we show that this estimate still holds if the  solution is bounded from above by a quadratic function in the tangent direction. We also obtain an existence result for the  convex solution of   Monge-Amp\`ere equations with Robin oblique
      boundary  conditions. 
\end{abstract}

 \vskip 15pt
\noindent{\bf Key Words:} oblique boundary value problem, Monge-Amp\`ere equation, $C^{2,\alpha}$ regularity
 \vskip 15pt
\noindent {\bf AMS Mathematics Subject Classification}:  35J60,  35J96, 35J25.

\vskip 15pt

\noindent {\bf  Running head}:  Regularity for the oblique derivative problem of   Monge-Amp\`ere equation

\vskip20pt

\baselineskip=15.8pt
\parskip=3pt

\newpage

\centerline {\bf Global $C^{2, \alpha}$ Regularity for the Oblique  Boundary Value Problem of  Monge-Amp\`ere Equations        }

 \vskip10pt

\centerline { Huaiyu Jian\ \  and \ \ Xushan Tu}

\tableofcontents

\baselineskip=15.8pt
\parskip=3.0pt

\newpage
\section{Introduction}\label{sec1}

In this paper, all solutions to the Monge-Amp\`ere equations are convex functions. We always assume that   $n\geq 2$, $\alpha \in (0,1)$ is a constant, $\Omega$ is a bounded convex domain in $\R^n$, and  $ 0<\lambda \leq f(x) \leq \Lambda$ on $\Omega$, and $\nu$ is the inward normal vector field  on $\partial \Omega$ if $\Omega \in C^1$.  Lions, Trudinger and Urbas \cite{[LTU]} studied the Neumann problem of the Monge-Amp\`ere equation
\[\det  D^2u  =f(x)      \text{ in } \Om  , \quad
D_{\nu}u  = \phi(x)                 \text{ on } \pom.\]
 Under the assumptions of $f$, $\phi $ and $\partial \Omega$ are sufficiently smooth, and $\Omega$ is uniformly convex, they proved the existence and uniqueness of smooth solutions in the space $C^{3, \alpha}(\bar \Omega)$ $(0<\alpha<1)$. Subsequently, Urbas \cite{[U87]} and Wang \cite{[W]} investigated the corresponding oblique derivative problem
 \begin{equation}\label{eq:oblique eq 1}
 	\det  D^2u  =f(x)      \text{ in } \Om  , \quad D_{\beta}u  = \phi(x)                 \text{ on } \pom.
 \end{equation} 
In two-dimensional case,  when $\Omega$ is $C^{2, 1}$ and uniformly convex $f\in C^{1, 1}(\bar \Omega)$,  $\phi \in C^{1, 1}(\partial \Omega)$, and $\beta \in C^{2, 1}(\partial \Omega; \R^n)$ is a small $C^1$ perturbation of the inner vector field $\nu$.
Urbas \cite{[U95]} proved that the solution belongs to $C^{2, \alpha}(\bar \Omega)$, and this was extended to higher dimensions in \cite{[U98]} under the further assumption that $\Omega$ is uniformly convex and $C^{3,1}$ smooth. The method of \cite{[LTU]} can also be applied to the oblique derivative problem for other fully nonlinear equations,  such as the oblique derivative problem for Hessian equation studied by Ma and Qiu \cite{[MQ]}, and for the augmented Monge-Amp\`ere equations studied by Jiang and Trudinger \cite{[JTX]}.

We study the pointwise Schauder regularity at the boundary of the Aleksandrov solution to \eqref{eq:oblique eq 1}, which means we investigate the $C^{2, \alpha} $ regularity of the solution when the positive coefficient $f$ is only in $ C^{\alpha}(\bar \Omega)$. The interior Schauder regularity of strictly convex solutions is obtained by Caffarelli  \cite{[C2]}, and a simplified proof can be found in Jian and Wang \cite{[JW]}. The boundary Schauder regularity for the Dirichlet problem is given by Trudinger and Wang in \cite{[TW]}, and a pointwise Schauder estimate is provided by Savin in \cite{[S]}. The global estimate for the second boundary value problem of the Monge-Amp\`ere equation in optimal transportation can be found in works by Delanoë \cite{[De]}, Urbas\cite{[U97]} and Caffarelli \cite{[C4]}. Recently, Chen, Liu, and Wang  \cite{[CLW]} made substantial improvements to the global  $C^{2, \alpha} $  regularity result for the natural boundary value problem in optimal transport in \cite{[C4]}.

When $n \geq 3$, the boundary Schauder regularity of \eqref{eq:oblique eq 1} is not naturally established, as shown in 
 \cite{[W]} or Example \ref{exam:pogorelov function} in this paper. Following the study of the Dirichlet problem in \cite{[S]}, we additionally assume that $u$ has quadratic growth at the given boundary point $x_0$. That is, there exists a constant $C_0>0$
and a subdifferential $p_{x_0} \in \partial u(x_0) $  such that 
\begin{equation}\label{eq:qg ch1}
	u(x)-u(x_{0})-p_{x_0} \cdot (x-x_{0}) \leq C_0|x-x_{0}|^{2} , \quad \forall x \in \partial \Omega.
\end{equation} 
 An example that satisfies the quadratic growth condition is when $\partial \Omega$ is $C^{1,1}$ and $x_0$ is the maximum value point of $u$.


Based on the convexity of the solution $u$, we interpret the equation $D_{\beta}u=\phi$ in \eqref{eq:oblique eq 1} in terms of the Dini derivative (see \eqref{eq:dini def}).

\begin{Theorem} \label{thm: c2alpha n=2}
	Let $n=2$ and $u \in C(\bar \Omega)$ be a solution of  \eqref{eq:oblique eq 1}. Suppose that $\pom  \in  C^{1,\alpha}$,    $f \in C^{\alpha}( \bar \Omega)$, $\phi \in C^{1,\alpha}(\bar{\Om})$,   $\beta \in C^{1,\alpha}(\partial \Omega; \R^n)$ is an oblique vector field\footnote{$\beta$ is oblique (point inward) at $x_0$ if $\left\{x_0+t\beta(x_0) |\, t> 0\right\} \cap \Omega \neq \emptyset$.}. Then, $u\in C^{2,{\alpha} }\left(\bar{\Omega}\right)$.
\end{Theorem}

\begin{Theorem} \label{thm: c2alpha n=3}
	Let $n\geq 3$ and $u \in C(\bar \Omega)$ be a solution of  \eqref{eq:oblique eq 1}. Suppose that $\pom  \in  C^{1,\alpha}$,    $f \in C^{\alpha}( \bar \Omega)$, $\phi \in C^{1,\alpha}(\bar{\Om})$,   $\beta \in C^{1,\alpha}(\partial \Omega; \R^n)$ is an oblique vector field.  If \eqref{eq:qg ch1} holds, then   $u\in C^{2,{\alpha} }(x_0)$.
\end{Theorem}

\begin{Remark}
	If we replace the assumption \eqref{eq:qg ch1} in Theorem \ref{thm: c2alpha n=3} with 
		\begin{equation}\label{eq: qg ch1 e}
		u(x)-u(x_{0})-p_{x_0} \cdot (x-x_{0}) \leq C_0|x-x_{0}|^{2} +\varepsilon, \quad \forall x \in \partial \Omega
	\end{equation} 
where  $\varepsilon>0$ is a sufficiently small constant depending on some universal constants, the conclusion still holds. In particular, the solution is actually a $C^{2,{\alpha} } $
 function around  $x_0$.
\end{Remark}

Obviously, Theorem \ref{thm: c2alpha n=2} can be seen as the extension of the Schauder regularity for Neumann problems of  Poisson equations in two dimensions. When $n \geq 3$, Theorem \ref{thm: c2alpha n=3} actually implies  that if $u$ is $C^{1,1}$ in all tangential directions, then $u\in C^{2,{\alpha} }\left(\bar{\Omega}\right)$. We will prove in Section \ref{sec:quasi sconx} that the quadratic growth condition \eqref{eq:qg ch1} can ensure the strict convexity, and the strict convexity of the solution, especially in the tangential direction, which is crucial for our boundary Schauder regularity.  As shown in Example \ref{exam:c11}, we can always construct non-strictly convex solution of the oblique derivative problem, which are $C^{1,1-\epsilon}$ along the tangent direction.

Our main tools are normalization and perturbation methods, with key points being the pre-compactness of the oblique derivative problem under normalization, the pre-compactness of the normalized solution, and the compactness of the corresponding oblique boundary values under convergence. 

For the normalized oblique derivative problem \eqref{eq:oblique eq 1}, our approach is to study the section {\sl    section $S_{t,\nabla u(x_0)}^u(x_0)$ of $u$ with height $h$, based point $x_0 \in \partial \Omega$  and a specific subgradient $p \in \partial u(x_0)$}, defined  by
\[S_{h,p}^{u}(x_{0}) :=\left\{x \in \bar{\Omega} |\ u(x)<u(x_{0})+p \cdot(x-x_{0})+h\right\}, \]
and its {\sl oblique boundary} $G_{h}^{u}(x_0): =  S_{h}^{u}(x_{0})  \cap \partial \Omega$. In Section \ref{sec:gsl}, we choose a special $p=\nabla u(x_0)$ and write $S_{t,\nabla u(x_0)}^u(x_0)$ as $S_{t}^u(x_0)$ or $S_{t}(x_0)$. We will prove that
$S_{h}^{u}(x_0)$ shrinks to $\left\{x_0\right\}$, and there exists a positive constant $c$ such that for every small $h$, $S_{h}^{u}(x_{0})$  satisfies
\[	c h^{\frac{n}{2}} \leq	\left|S_{h}^{u}(x_{0})\right| \leq c^{-1} h^{\frac{n}{2}}\]
and 
\[	c\P^{x_0} S_{h}^{u}\left(x_{0}\right)+(1-c)\P^{x_0}x_0 \subset \P^{x_0} G_{h}^{u}\left(x_{0}\right) \cap  \P^{x_0}\left(2x_0- G_{h}^{u}\left(x_{0}\right) \right),  \] 
where 
 $\P^{x_0}=\P^{x_0,\beta(x_0)} $ denotes the projection mapping along the direction $ \beta (x_0) $  to the tangent plane $\partial \Omega$ of  at $x_0$.
We refer to this as the good shape property of $S_h^{u}(x_0)$, which is invariant under convergence and any linear diagonal transformation  $\D$ that keeps  $\beta$ and the tangent plane invariant. As a result, the normalized oblique derivative problems are pre-compactness.

In the case $n=2$,  the good shape properties holds at any boundary point, the engulfing property of sections is naturally induced, and it can be shown that the solution is always $C^{1,\alpha}$ up to the boundary. However, when $n \geq 3$,  there are two obstacles, stemming from the lack of compactness of solutions and the corresponding oblique boundary values.

	To ensure the compactness of the normalized solution, we prove that the solution with zero oblique boundary values is {\sl uniformly strictly convex} near the given boundary point.   In Section \ref{sec:uscl}, we will use this fact and the good shape property to construct a universal strictly convex modulus for solutions of normalized oblique derivative problems,  which is equivalent to a pointwise uniform $C^{1,\alpha} $ estimate at the given boundary point $x_0$. 

In Section \ref{sec:viscosity subsolution},  we study the compactness of oblique boundary values. By combining the interior equation and the convexity of the domain, we investigate the upper semicontinuous (convex) viscosity subsolutions of \eqref{eq:oblique eq 1}. Several theorems on existence and compactness are presented here.

In Section \ref{sec:liouville and station}, we will study the blow-up limit and prove that the normalized oblique boundary $\tilde{G}_h$ tends to flat, so that the blow-up limit corresponds to  solutions to the Neumann problem for the Monge-Amp'ere equation in the half-space under some specific restrictions. 
\begin{Theorem} [{\cite[Theorem 1.1]{[JT]}}]\label{thm:liouville thm}
	Let $u \in C(\overline{\R_+^n})$ be a convex solution of
	\[ 	\det  D^2u  =1       \text{ in }\R_+^n , \quad
	D_nu = ax_1                   \text{ on } \partial\R_+^n. \]
	If $n=2$, or if $n\geq 3$ and $a=0$, or  $n\geq 3$ and $u$ satisfies $		\lim_{z\in \R^{n-2}, z \to \infty}  \frac{u(0,z,0)}{|z|^2 } <  \infty $, then $u$ is a quadratic polynomial.
\end{Theorem}
Then a standard perturbation methods as in \cite{[JW]} gives the $C^{2,\alpha}$ estimate.
In Section \ref{sec:example}, we will list some examples that illustrate the obstacles that may arise in oblique derivative problems.

\vspace{8pt}

\section{Preliminaries} \label{sec2}

In this paper, A point in $\R^{n}$ is written as \[x=(x_1,\cdots,x_{n-1}, x_n)=(x',x_n),\]
and if $n\geq 3$, it will also be written as $x=(x_1,x'',x_n)$.  Define  $\P x=x'$, {\sl the projection mapping along the $x_n$-axis   to the  hyperplane $\R^{n-1}:=\left\{x_n=0\right\}$}, and
set
\[\R_+^n= \left\{x=(x_1,\cdots,x_{n-1}, x_n)\in \R^n|\, x_n>0\right\}.\]
Denote by $\I$  [or $\I'$, $\I''$] the identity matrix of size $n$ [or $n-1$, $n-2$]
and $B_r(x)$  [or $B_r'(x')$, $B_r''(x'')]$  the ball of radius $r$  centered at $x$ [or $x'$, $x''$] in $\R^n$ [or $\R^{n-1}$,  $\R^{n-2}$].

The term ``universal constant" refers to a positive constant that depends only on  $diam(\Omega)$ and $ \left\|\partial \Omega\right\|_{C^{1,\alpha}}$, $\alpha$, $n$, $\lambda$, $\Lambda$,   $\left\| \phi\right\|_{ Lip}$, the oblique constant $\eta $ of $\beta$, $\left\| \beta\right\|_{ Lip}$,  the quadratic constant $C_0$ whenever they are involved,  but not on the specific functions $u$ and $\upsilon$.  We denote these universal constants by $c$ if small or $C$ if large.  Moreover, we denote $C_q$ as a positive quantity that depends only on the universal constants and the parameters $q$.  Given two non-negative quantities $ b_1$ and $b_2 $, we will use the notation 
\begin{equation}\label{eq:symbol}
	b_1 \lesssim b_2
\end{equation}
to indicate that $ b_1 \leq Cb_2 $ for a universal constant $C$.  We also denote 
\[b_1 \approx b_2 \] 
if  $ b_1 \lesssim b_2 $  and $b_2 \lesssim b_1$.  We will write  $\A \leq \B$ for matrices $\A$ and $\B$ when $\B-\A$ is positive semidefinite. Similarly, we use $\A \lesssim \B$ for two positive semidefinite matrices $\A$ and $\B$  to indicate that  $\A \leq C\B $, and use $\A \approx \B$ if  $\A \lesssim \B $  and $\B \lesssim \A$. The symbol $ \gtrsim$ is understood in a similar way. If the constant $C$ is replaced by $C_q$, then we use the symbols $\lesssim_q$, $\gtrsim_q$ and $\approx_q$ instead.

 In this paper, we need some preliminaries on convex solutions to Monge-Amp\`ere equations. Readers can refer to  related chapters in   \cite{[C1],[C2],[C3],[F],[G], [TW1]}. 

	For $E, F \subset \R^n$ and $\kappa \in \R$ we define 
\[E+F:=\left\{ x+y|\, x\in E, y \in F\right\} , \quad \kappa E= \left\{\kappa x |\, x \in E \right\} ,\] 
and we write $-E $ for $(-1)E$, $E-F$ for $E+(-F)$ and $E+x $ for $E+\left\{ x\right\}$, where $x \in \R^n$.
The Lebesgue measure of a measurable set $E\subset \R^n $ is denoted by $|E|$.

For  $\kappa>0$, a bounded convex set $E\subset \R^{n}$ is said to be $\kappa$-balanced about a point  $x$ if
\[t(x-E)\subset E-x  \text{ for all } t \in [0,  \kappa].  \] 
If $\kappa $ can be taken to be universal,  then we say that $E$ is balanced about $x$.

	\begin{Lemma}[John's Lemma]\label{lem2.9 John's Lemma}
	Suppose  $S \subset \R^n$ is a bounded convex set and $\mathring{S} \neq \emptyset$. Then there is an  ellipsoid  $E$ (called the John ellipsoid, the unique ellipsoid of the maximal volume contained in $S$) such that  
	\[ E  \subset  S  \subset C \left(E-x_E\right)+x_E,\]
	where $ x_E$ is the mass center of $E$, and $C$ depends only on $n$.
	\end{Lemma}

	\begin{Lemma}[{\cite[Lemma 2.4]{[JT]}}]\label{lem2.5 balance lemma}
	Given   convex set $E \subset \R^n$ and  point $x\in \R^n$.  Suppose that the line $x+ te_n : t \in \R$ intersects $E$ at the points $y$ and $z$.  Then, we have 
	\[|y_n-z_n| \cdot	\operatorname{Vol}_{n-1}\P E  \lesssim |E|.\]
	Moreover, if $\P E$ is $ \kappa$-balanced about $\P z$, then we have the inverse inequality
	\[|y_n-z_n| \cdot	\operatorname{Vol}_{n-1}\P E  \gtrsim \kappa |E|.\]
\end{Lemma}

The subdifferential of a convex function $u$ at $x_0 \in \bar \Omega$ is  
\[
\partial u(x_0)=\left\{p \in \mathbb{R}^{n} |\, u(x) \geq u(x_0)+p \cdot(x-x_0) \text{ for all } x \in U\right\}.
\]  
Each element $p \in \partial u(x_0)$ is called a subgradient of $u$ at $x_0$, and the function $u(x_0)+p \cdot(x-x_0) $ is a supporting hyperplane function of $u$ at $x_0$.  
We denote $\partial u(E):=\bigcup_{x \in E} \partial u(x)$ for  $E \subset \subset \Omega$, which is measurable if $E$ is open or closed, and hence is measurable if $E$ is Borel. 
The Monge-Amp\`ere measure $\M u$ is defined by 
\[\M u(E)=|\partial u(E)| \text{ for each Borel set } E \subset\subset U,\]	
we say $u$ is a (generalized/Aleksandrov) solution to $\det D^2 u = f $ if $\M u= fdx$.  We will use the following comparison principle for generalized solutions to Monge-Amp\`ere equations, which can be found in, e.g.,  \cite{[F]} or \cite{[G]}. 

\begin{Lemma}\label{lem:comp pD}
	Suppose $\Omega$ is a bounded convex set, and $u \in C\left(\overline{\Omega}\right)$ and $v \in C\left(\overline{\Omega}\right)$ are convex functions. If $	\det D^2 u \leq \det D^2 v \text{ in } \Omega, \quad u \geq v\text{ on } \partial \Omega$,	then $u \geq v$ in $\Omega$.
\end{Lemma}

	As a corollary,
	\begin{Lemma}\label{lem:comp pN}
		Suppose $\Omega$ is a bounded convex set,  $G_1$ and $ G_2$ are closed sets on $\partial \Omega$ with no interior intersection and satisfy	$G_1 \cup G_2  =\partial \Omega  $,  and $\beta$ is oblique on $ \partial \Omega \setminus G_1 $.
		Let $u \in C\left(\overline{\Omega}\right)$ and $v \in C\left(\overline{\Omega}\right)$  be two  convex functions  that satisfy  
		\begin{equation} \label{eq3.8 comparison equation 1}
			\begin{cases}
				\det  D^2v \leq   \det  D^2u      & \text{ in } \Omega , \\
				u <  v                   &\text{ on }    G_1 , \\
				D_{\beta}v <  D_{\beta}u                     & \text{ on } G_2.
			\end{cases}
		\end{equation}
		Then, $u <  v  \text{ in } \bar{\Omega}$.
	\end{Lemma}

%
%
%
%
%
%
%
%
%

	The following lemmas are well known for convex functions and are corollaries of the Aleksandrov-Bakelman-Pucci maximum principle, see \cite[Theorem 1.4.2]{[G]}.
\begin{Lemma}[Aleksandrov’s Maximum Principle]\label{lem:bdy c1/n}
	Suppose $u \in C(\overline{\Omega})$ is convex, $u=0$ on $\partial \Omega$. There exists $C>0$, depending only on the dimension $n$, such that
	\[
	u(x) \geq-C\left([\operatorname{diam}(\Omega)]^{n-1}\operatorname{dist}\left(x, \partial \Omega\right) \M u(\Omega)\right)^{\frac{1}{n}}.
	\]
\end{Lemma}	

By comparison with an appropriate quadratic function, we can obtain from Lemma \ref{lem:comp pD} that 
	\begin{Lemma}\label{lem:sec vol lbd} 
	Let $u$ be a convex function on $\Omega$ satisfying $\det  D^2u > \lambda $ in $ \Omega$, then $|\Omega| \leq C(\lambda)\left\| u\right\|_{L^{\infty}(\Omega)}^{\frac{n}{2}} $. 
	In particular, we always have $\left|S_{h,p}^u(x_0)\right|\lesssim h^{\frac{n}{2}}$.
\end{Lemma}


\section{A Qualitative Strict Convexity Lemma}\label{sec:quasi sconx}

In this section, we introduce the qualitative strict convexity lemma for solutions of  Monge-Amp\`ere equations. We first review some lemmas about the strict convexity of convex functions, which can be found in  \cite{[C1],[C3]}.
\begin{Lemma} 
	Assuming $n=2 $,   $\Omega \subset \R^2$ is a  convex domain,  $u\in C( {\Omega})$ is convex, and
	\[ \lambda \leq \det  D^2u \leq  \Lambda \text{ in }  \Omega.\]
	Then $u$ is strictly convex and $C^1$ in $\Omega$.
\end{Lemma}

\begin{Lemma} \label{lem2.15 strictly convex  pogorolev}
	Suppose that $n \geq 3$, $\Omega \subset \R^n$ is a  convex domain, $u\in C(\bar {\Omega})$ is convex function, and
	\[ \lambda \leq \det  D^2u \leq  \Lambda \text{ in }  \Omega,\quad u=\varphi \text{ on } \partial \Omega. \]
	If $\varphi \in C^{1,\alpha}(\partial \Omega)$ for $\alpha > 1-\frac{2}{n}$, then $u$ is strictly convex and $C^1$ in $\Omega$.
\end{Lemma}

\begin{Lemma} \label{lem2.16 strict classical line lemma}
	Suppose that $u\in C(\Omega )$  is convex and satisfies
	$0 < \lambda \leq \operatorname{det}  D^2u \leq \Lambda $.
	For $x\in \Omega$ and $p\in \partial u(x)$, we define
	$  \Sigma:=\left\{y|\,u(y)=u(x)+p\cdot (y-x)\right\}$.
	Then, either $\Sigma=\left\{x\right\}$ or $\Sigma $ has no extreme point\footnote{A point is an extreme point of a  convex set $E$ if it does not lie in any open line segment joining two points of $E$.} inside $\Omega$.
\end{Lemma}

As is well-known, the $C^{1, \alpha}$ and  $C^{2, \alpha}$ interior estimates of strictly convex solutions to the Monge-Amp\`ere equation were established by Caffarelli \cite{[C1],[C2],[C3]}. Also see the  related chapters in  \cite{[F], [G]}.

 In this section,   we always assume that $0 \in \partial \Omega$ and 
\begin{equation} \label{eq2.6 boundary prigin definition}
	\Omega \cap  B_{c}(0):=\left\{(x', x_n) |\, x_n > g(x'),\ x'\in B'_c(0) \right\} ,
\end{equation}
such that the part of the boundary $\partial \Omega $ near $0$  is described by a locally Lipschitz function $g$ (or $ C_{loc} ^ {1, \alpha} $ if $\partial \Omega \in C^ {1, \alpha}$) as
\begin{equation}\label{eq2.7 boundary definition}
	G:=\left\{(x',x_n) |\, x_n = g(x'),\ x'\in B'_c(0)\right\}.
\end{equation}
Since $\Omega $ is convex,  $g$ is convex  and satisfies
\begin{equation} \label{eq2.8 boundary g definition}
	0 \leq  g(x') \leq C|x'|,\quad \forall x'\in B'_c(0).
\end{equation}
We use $\G x' = \G (x',0)$ to denote the point $(x',g(x'))$, where
We will use the following notation,
\begin{equation}\label{eq2.9 boundary map  definition}
	\G x= 	\begin{cases}
		(x',x_n)   &   \text{ if } x_n \geq g(x') ,  \\
		(x',g(x'))         &  \text{ if } x_n \leq g(x').
	\end{cases}
\end{equation} 
\begin{Lemma} \label{lem:quasi sconx}
	 Let $u\in C(\bar{\Omega})$ be a convex function,  $u\geq 0$  and satisfies
	\[	0 < \lambda \leq \det   D^2u \leq \Lambda  \text{ in }  \Omega.\]
	Then there exists a module \footnote{In this paper, a module $\sigma_i(t)$ represents a non-negative, strictly increasing, and continuous function that depends on the parameter $t$ and satisfies $\sigma_i(0)=0$.} $\sigma_{1} $ such that if  $\left\|u\right\|_{L^{\infty}\left(\Omega \cap B_c\right)} \leq K $  and
	\begin{equation} \label{eq3.1 stricty convex balance n-2 condition}
		\frac{\operatorname{Vol}_{n-1}\left(\P \check{S}_h(0) \cap (-\P \check{S}_h(0)) \cap B_1'(0) \right)}{h^{\frac{n-2}{2}}}  \geq K^{-1}, \quad \forall h\geq \sigma_{1}  ,
	\end{equation}
	where $\check{S}_h(0):=\left\{ x\in \Omega| u(x) \leq h\right\}$, then $u(0) \geq \sigma_{1} (K^{-1})$.  	
\end{Lemma}
\begin{proof}
	Assume that  $u(0) \leq \sigma_{1}$ (to be determined) is very small. For any  $h \geq C \sigma_{1}$,
	let $a_h =\sup \left\{ t|\, te_n \in \check{S}_h(0)\right\}$. Note that $u(he_n) \leq hu(e_n)+(1-h)u(0) \leq Kh+\sigma_{1} $, so we have   $a_h \geq ch$, where $c$ depends on $K^{-1}$. On the other hand, by Lemma \ref{lem2.5 balance lemma} and  Lemma \ref{lem:sec vol lbd}, we have
	\[	a_h\operatorname{Vol}_{n-1}\left(\P \check{S}_h(0) \cap (-\P \check{S}_h(0))  \right) \lesssim a_h\operatorname{Vol}_{n-1}(\P \check{S}_h(0)) \lesssim |\check{S}_h(0)| \lesssim h^{\frac{n}{2}},\]
	which together with \eqref{eq3.1 stricty convex balance n-2 condition} implies  that $   a_h \leq Ch$. Therefore, $ ch \leq a_h \leq Ch$. For each    $t\geq C\sigma_{1}$, we now consider the non-increasing function
	\begin{equation} \label{eq3.2 equation qsc1l e1}
		b(t):=\frac{u(te_n)}{t}\approx 1.
	\end{equation}
	We claim that {\sl  there exists  $\delta\in (0, 1) $ depending on $K$ such that
		\begin{equation} \label{eq:proof qlst 1}
			b(t/2) < (1-\delta)b(t), \quad \forall t \geq \delta^{-1} \sigma_{1}.
	\end{equation}  }

	In contrast to \eqref{eq:proof qlst 1}, suppose there exists a positive constant $t_0\geq C\delta^{-1} \sigma_{1}$  such that $b({t_0}/2) \geq (1-\delta)b({t_0})$. For simplicity, we can assume that
	\[b({t_0})=1 \text{ and } b({t_0}/2) \geq 1-\delta.\]
	Let
	\[ E:=\left\{(x', x_n)\in \check{S}_{\delta {t_0}}(0)|\, x'\in   \P \check{S}_{\delta {t_0}}(0)\cap (- \P \check{S}_{\delta {t_0}}(0))\right\}.\]
	 Then $ \P E$ is balanced about $0$, by Lemma \ref{lem2.5 balance lemma}, we have 
	\begin{equation}\label{eq3.4 equation qsc1l e3}
		0\leq \sup \left\{ x_n| x=(x, x_n) \in E\right\} <\frac{{t_0}}{4}.
	\end{equation}
	 Now, let $v(x)=u(x)-x_n$ and take the points $Y_1=t_0e_n $ and $ Y_2=\frac{{t_0}e_n}{2}$. We have,
	 \[ -\delta t_0 \leq v(Y_1) \leq 0,\quad -\delta t_0 \leq v(Y_2) \leq 0,\quad  \text{  and } v\leq u \leq \delta t_0 \text{ in } E.\]
Consider the cone domains
	\[  \Gamma_1^+:=\left\{ Y_1+s\left(z-Y_1\right)| s \in \left(0,1\right) ,  z \in  \P E\right\}, \text{ and }  \Gamma_2^-:=\left\{ Y_2-s(z-Y_2)| s \in \left(1/2,2\right) ,  z \in \P E\right\}. \]
	By convexity, we have
	\[  -C\delta t_0 \leq  v \leq C\delta t_0 \text{ in } \Gamma_1^+ \cap  \Gamma_2^-.\] 
	Then Lemma \ref{lem:sec vol lbd} implies $\left|\Gamma_1^+ \cap  \Gamma_2^-\right|\lesssim(\delta {t_0})^{\frac{n}{2}}$, which contradicts 
	\[ \left|\Gamma_1^+ \cap  \Gamma_2^-\right|\geq c{t_0}\operatorname{Vol}_{n-1}(\P E) \geq  c K^{-1} \delta^{-1} (\delta {t_0})^{\frac{n}{2}}  \]
 provided $\delta$  is small. Thus, we have proved \eqref{eq:proof qlst 1}.
	By iteration, \eqref{eq:proof qlst 1} implies
	\[	b(t) \lesssim t^{|\log_2(1-\delta)|} \text{ for } t \gtrsim  \delta^{-1} \sigma_{1},\]
	which contradict \eqref{eq3.2 equation qsc1l e1} when $\sigma_{1}$ is small.  

\end{proof}

\begin{Lemma} \label{lem:scon 2dires}
	%
	 Let $u\in C(\bar{\Omega})$ be a convex function,  $u\geq 0$,  and $	0 < \lambda \leq \det   D^2u \leq \Lambda $ in $\Omega$.
	 Assume that $u \lesssim 1$ in $\Omega$ and  satisfies
	\begin{equation} \label{eq:proof qlst 2}
		u\left(x\right)\lesssim |x''|^2   \text{ on }  G \cap \left\{x_1=0\right\},
	\end{equation}
	then there exists a modulus  $\sigma_{2} $ such that $	u(x) \geq \sigma_{2}(|x_1|+|x_n|)$ in $\Omega \cap B_{c}(0)$. 
\end{Lemma}

\begin{proof}
	Fix  $\rho>0$ small and take a point  $y = \left( y_1,y'' ,y_n\right)\in \Omega $ satisfying $|y_1|+y_n=\rho$. Our goal is to get a lower bound $u(y)$ in terms of $\rho$. We can simplify the problem by using the reflection transformation with respect to the $e_1$ axis and the sliding transformation $\A x:=\left(x_1,x''-\frac{(x_1+x_n)}{\rho}y'',x_n\right)$. Therefore, without loss of generality, we can assume that $y=y_1e_1+y_ne_n$.   Then, we apply a rotational transformation $\B$  on the plane spanned by $e_1$ and $e_n$, which transforms the vectors $y_1e_1+y_ne_n$ and $y_n e_1-y_1e_n$ into $\rho e_1$  and $\rho e_n$ respectively. After that, we consider a translation transformation $\T$ that moves the point $\frac{1}{2}\rho e_1$ to the origin. 	Let $v(x) =u\left(\B^{-1}  \T^{-1} x \right)$ and
	\[\Omega_{v}:=\T \B \left(\left\{ t y + (1-t)z|\, z \in   \Omega \cap \left\{x_1=0\right\}, t\in (0, 1)\right\}\right), \]
	then
	\[0\leq v \lesssim 1 \text{ and }  \det D^2v \approx_{\rho}  1 \text{ in } \Omega_v. \]
	By locally writing
	\[G_{v}:=\T \B\left(\left\{ t y + (1-t)z|\, z \in G \cap \left\{x_1=0\right\}, t\in (0, 1)\right\}\right)=\left\{(x',x_n) |\, x_n = g_{v}(x')\right\} , \]
	 we can see that  $g_{v}$ is convex in $ B_{\rho^4}'(0)$ and satisfies $ 0\leq  g_{v}(x') \lesssim_{\rho} |x'|$. Combining convexity and the assumption \eqref{eq:proof qlst 2}, we have
	\[	v\left(x\right)\lesssim \rho^{-1} |x''|^2 +u(y)  \text{ on }  G_{v} \cap B_{\rho^4}^+(0),\]
	 which implies that $\check{S}_h: =\left\{ x\in  \Omega_{v}|\, v(x)\leq h\right\}   $ satisfies
		\[\frac{\operatorname{Vol}_{n-1}\left(\P^{\nu} \check{S}_h \cap (-\P^{\nu} \check{S}_h) \cap E \right)}{h^{\frac{n-2}{2}}}  \gtrsim_{\rho} 1\text{ for } h\geq 2u(y).\]
 	 Then, a modified version of Lemma \ref{lem:quasi sconx} gives a lower bound for $v(0)$, which in turn gives a lower bound for $u(y)$.

\end{proof}

\vspace{8pt}

\section{Global Lipschitz Regularity of solutions}\label{sec:Lipschitz}

If $\partial \Omega \in C^1$,  $\beta$ is oblique (point inward) on $\partial \Omega$ means that $\beta \cdot \nu \geq \eta>0$ for some $\eta$, where $\nu$ is the inward normal vector field.
 If $\Omega$ is only a bounded convex domain, we say that $\beta$ is oblique (point inward) at $x_0$ provided $\left\{x_0+t\beta(x_0) |\, t> 0\right\} \cap \Omega \neq \emptyset$. Since $\beta$ is continuous, by compactness we can assume the existence of constant $\eta$, called the oblique constant of $\beta$, depending on $\left\|\beta\right\|_{C^0}, \left\|\partial \Omega\right\|_{Lip}$,  such that 
\[\eta\leq \left\|\beta\right\| \leq \eta^{-1}, \quad \text{ and }\left\{ x_0+ t\left(\beta(x_0) +B_{\eta}(0)\right)  |\, t \in \left(0,\eta\right)\right\} \subset \Omega.\]
Combining the convexity of  $\Omega$ and the oblique property  $\beta$, the set 
\[\left\{y+t\beta (y)|\, y \in \partial \Omega , t > 0\right\}\] covers a neighborhood of  $\partial \Omega$. 

 Let $u $ be a solution of \eqref{eq:oblique eq 1}. Locally, after translation and rotation, we assume that the boundary point $x_0=0$, and 
 \[\Omega \subset \left\{ x_n \geq 0 \right\} \text{ satisfies \eqref{eq2.6 boundary prigin definition}-\eqref{eq2.8 boundary g definition}.}\]
 Therefore, we always consider $\phi $ and $\beta $ as functions depending only on the variable $x'$, and use $\phi(x')=\phi(\G x')$ and $\beta(x')=\beta(\G x')$ to denote them, respectively.

The Dini derivative of $u$ along any direction $\gamma$ at $x$, where $\gamma$ is oblique if $x \in \partial \Omega$, is denoted by
\begin{equation}\label{eq:dini def}
	D_{\gamma} u(x):=D_{\gamma}^+	u(x) =\lim_{t\to 0^+} \sup_{p \in \partial u\left(x+t\gamma (x)\right)} p \cdot \gamma \in [-\infty, +\infty).
\end{equation}
If $\gamma$ is continuous, then $	D_{\gamma} u$ is an upper semicontinuous function in $\Omega$. 
If $u$ is Lipschitz near $x_0$, which is always true for $x_0 \in \Omega$,  then we have 
\begin{equation} \label{eq:Dini supgradient}
	D_{\gamma}	u(x) = \sup_{p \in \partial u(x)} p \cdot \gamma=\lim_{t\to 0^+}D_{\gamma}	u(x+t\gamma (x)).
\end{equation}

 We  assume  $\beta(0)=e_n$ by considering the function   $u\left(\B y\right)$. Here, the sliding transformation $\B  x=\beta_nx+\sum_{i=1}^{n-1}\beta_i x_n e_i $ is defined by the expression $\beta (0)=\left(\beta', \beta_n\right)$. By proving that $u \in Lip \left(\bar \Omega\right)$, we locally replace the oblique equation $D_{\beta}u =\phi$ with the Dini derivative $D_{n}u$. We let $\gamma =e_n$ and choose  $\nabla u(x)$ such that \eqref{eq:Dini supgradient} holds, where such a definition is not unique and depends on the oblique vector $\gamma$ near $x_0$.   Additionally, we use  $\nabla u(x)$ to denote any element in   $	\partial u(x) $  for any $x \in \Omega $. The choice of $\nabla u(x)$ does not affect the proof. Finally, by subtracting the support function $u(0)+\nabla u(0)\cdot x$, we can assume that
 \begin{equation}\label{eq3.14 local problem assump 2}
 	u(0)=0,\quad  \nabla u(0)=0,\quad u\geq 0 \text{ in } \Omega,\quad  \beta(0)=e_n,\quad \phi(0)=0.
 \end{equation}
 For simplicity, we will always write    $ S_{h,\nabla u(0)}^{u}(0)$ and $ G_{h,\nabla u(0)}^{u}(0)$ as $S_{h} $ and  $ G_{h}$, respectively.

To provide a Lipschitz bound, we now demonstrate how to control the gradient in the remaining directions from the oblique direction.
\begin{Lemma}\label{lem:lip on modules}
	 For any convex domain  $\Omega$  satisfying \eqref{eq2.6 boundary prigin definition}-\eqref{eq2.8 boundary g definition}, assume that $u\in C\left(\overline{\Omega  \cap  B_c^+(0)}\right)$ is convex, for any $ x\in \Omega  \cap \bar B_{c}^+(0)$ and  $p \in \partial u(x)$, we have 
\[ |p'|\lesssim  \frac{2\left\|u\right\|_{L^{\infty}\left(\partial \Omega \cap B_{2|x|}(0)   \right)}+|p_n|\left(x_n+|x'|\right)}{|x'|}\]
and 
	 \[ |p'| \lesssim \inf_{r \geq |x'|} \left(\frac{2\left\|u\right\|_{L^{\infty}\left(\Omega \cap B_{2r}(0)   \right)}}{r}-p_n\right).\] 
%
\end{Lemma}
\begin{proof}
	We may assume that  $p'\neq 0$. By choosing the point  $y \in \partial \Omega$ such that $y'= x'+ \frac{p'}{|p'|}|x'|$, we have $y_n \lesssim |y'| \lesssim |x'|$ and
	\[|p'|\cdot |x'| +|p_n|\left(x_n+|x'|\right)\lesssim |p'|\cdot |x'|+ p_n(y_n-x_n) = p\cdot (y-x) \leq u(y)-u(x) \leq \left\|u\right\|_{L^{\infty}\left(\partial \Omega \cap  B_{2|x|}(0) \right)}.\] 
	 For any $r>|x'|$, by choosing a point $z \in \Omega$ such that $|p'| \left(z-x\right)=r(cp',|p'|) $,  we have  
	 \[r(p_n+c|p'|)=p\cdot (z-x)  \leq u(z)-u(x) \leq \left\|u\right\|_{L^{\infty}\left(\Omega \cap B_{2|x|}(0)   \right)}.\]  
\end{proof}
 
\begin{Lemma}\label{lem:lipschitz bound dC0}
	 Let $u $ be a solution to \eqref{eq:oblique eq 1}. Suppose $x_0 \in \partial \Omega$, then there exists $p \in \partial u(x_0) $ satisfying \eqref{eq:Dini supgradient} for $\gamma =\beta$, denoted by $ \nabla u(x_0)$, such that
	 \[|\nabla u(x_0)| \lesssim  \omega_u(\Omega)  + \left\|\min\left\{\phi,0\right\}\right\|_{L^{\infty}(\partial \Omega)},\]
	 where $\omega_u(\Omega)$ denotes the oscillation of $u$ on $\Omega$. Therefore, 
	  \[\left\|u \right\|_{Lip(\bar \Omega)} \lesssim \omega_u(\Omega)  + \left\|\min\left\{\phi,0\right\}\right\|_{L^{\infty}(\partial \Omega)}.\]
\end{Lemma}
\begin{proof}
	Applying Lemma \ref{lem:lip on modules},  we obtain that for each $x_0 \in \partial \Omega$ and $t>0$ small, the slope of any support function of $u$ at point $x_t:=  x_0+t\beta$ is bounded by $C\left( \omega_u(\Omega)  +\left\|\min\left\{\phi,0\right\}\right\|_{L^{\infty}(\partial \Omega)}\right)$, and as $t\to 0^+$, they subconverge to some support function $\ell_{0}$ of $u$ at $x_0$. It can be verified that the gradient $p:=\nabla \ell_{0}$ satisfies \eqref{eq:Dini supgradient}. 	
	
	The proof is completed by applying Lemma \ref{lem:lip on modules} and noting that $\left\{ y+t\beta(y ) |\, y \in \pom, t \geq 0\right\}$  covers a neighborhood of the $\partial \Omega$ and $u$ is convex.
\end{proof}

\begin{Theorem} \label{lem:lipschitz bound id}
	Let $u $ be a solution to \eqref{eq:oblique eq 1}.  Suppose $\phi \in C(\partial \Omega)$, then 
	\[\omega_u(\Omega) \lesssim diam(\Omega)\left\|\min\left\{\phi,0\right\}\right\|_{L^{\infty}(\partial \Omega)}.\]
	 Hence, $\left\|u \right\|_{Lip(\bar \Omega)}\lesssim\left(1+diam(\Omega)\right) \left\|\min\left\{\phi,0\right\}\right\|_{L^{\infty}(\partial \Omega)}$.
\end{Theorem}
\begin{proof} 
	We can assume that $u$ attains its maximum at the boundary point  $0$, $\beta(0)=e_n$ and $\Omega$  satisfies \eqref{eq2.6 boundary prigin definition}-\eqref{eq2.8 boundary g definition}. By writing $\nabla u(0) =\phi(0) e_n+be$ with unit vector $e \in \R^{n-1}$ and $b \geq 0$, we have
	\[ bt+  \phi(0) g(te_n) =\nabla u(0) \cdot \G(te) \leq u\left(\G(te)\right)-u(0) \leq 0, \quad \forall t>0 \text{ small.}\]
	Therefore, $b \lesssim |\min\left\{\phi(0),0\right\}| \leq  \left\|\min\left\{\phi,0\right\}\right\|_{L^{\infty}(\partial \Omega)}$ and then $|\nabla u(0)| \lesssim \left\|\min\left\{\phi,0\right\}\right\|_{L^{\infty}(\partial \Omega)}$, which implies that 
	\[\inf_{\Omega} u \geq  \inf_{\Omega}(u(0) + \nabla u(0)\cdot x )\geq \sup_{\Omega} u -Cdiam(\Omega)\left\|\min\left\{\phi,0\right\}\right\|_{L^{\infty}(\partial \Omega)}.\]
\end{proof}

%


\begin{Lemma} \label{lem:D_nu bound}
		Let $u $ be a solution to \eqref{eq:oblique eq 1}. Suppose $\phi\in C(\partial \Omega)$,
		  $0\in \partial \Omega$, $\beta(0)=e_n$, $\Omega$  satisfies \eqref{eq2.6 boundary prigin definition}-\eqref{eq2.8 boundary g definition}, and \eqref{eq3.14 local problem assump 2} holds. Then
	\begin{equation}\label{eq:nb value lip}
		|D_{n} u(\G x')- \phi(x')|  \lesssim \left\|u\right\|_{Lip(\bar \Omega)} |\beta(x') -e_n |\lesssim|x'| \text{ on } \partial \Omega \cap B_c(0).
	\end{equation}
 	 If we also assume that  $u(x) \lesssim |x|^{1+{\alpha_0}}$ on $ \partial \Omega$ for ${\alpha_0} \in [0,1]$, then we can improve \eqref{eq:nb value lip} to:
 	\begin{equation}\label{eq:nb value c11}
 		|D_{n} u(\G x')- \phi(x')| \lesssim   |x|^{1+{\alpha_0}}  \text{ on } \partial \Omega \cap B_c(0).
 	\end{equation}
\end{Lemma}	
\begin{proof}	
	Given a point $x=\G x' \in  \partial \Omega \cap B_c(0)$, based on \eqref{eq:Dini supgradient} and the definition of $\nabla u(x)$, we obtain that
	\begin{equation}
		\begin{split}
		D_{n} u(\G x')=	\sup_{p \in \partial u(x)} p \cdot e_n & \geq  \nabla 	u(x) \cdot e_n  \\ &
			\geq \nabla 	u(x)  \cdot \beta (x') -C|\beta(x') -e_n|\cdot\left|\nabla u(x)\right| \\ &
			\geq  \phi(x') -C\left\| u\right\|_{Lip(\bar \Omega)}|x'| .
		\end{split}
	\end{equation}

	Since $\beta \in Lip(\partial \bar \Omega)$ and $\beta(0)=e_n$,  for every fixed $x=\G x' \in  \partial \Omega \cap B_c(0)$ and $t>0$ small, we can write $ x+te_n= x_s+s\beta(x_s') $, where $x_s\in \partial \Omega$ and $s>0$ satisfying 
	\[ s \approx t, \quad  |\beta(x_s')-e_n| \leq \frac{1}{2}|x_s'| \text{ and } |x_s'|\leq 2|x'|.\]
	Then, by abuse the notation $\nabla$, we have
	\[ 	D_nu(x+te_n)  \leq D_{\beta(x_s')}u(x_s+s\beta(x_s') ) +C|\beta(x_s')-e_n| \cdot |\nabla u(x_s+s\beta(x_s'))|. \]
	
	Now consider the family of functions $\left\{\phi_{t}\right\}_{t\geq 0}$ defined by $\phi_{t}(y'):=	D_{\beta(y')} 	u\left(\G y'+t\beta (y')\right): B_c'(0) \to R$. $\phi_{t}$ are all upper semicontinuous and converge decreasingly to the continuous function $ \phi$ as $t \to 0^+$.  Therefore, this convergence is uniform and there exists a modulus $\sigma_{3}(t)$ (depending on $u$) with $\sigma_{3}(0)=0$, such that
	\begin{equation}\label{eq3.17 equation nbvl e1}
		\phi_{t}(y')\leq \phi(y') +\sigma_{3}(t).
	\end{equation}
	Hence,
	\[\begin{split}
		D_{n} u(\G x') & = \lim_{t\to0^+} D_nu(x+te_n)  \\
		& \leq \lim_{t\to0^+}[   D_{\beta(x_s')}u(x_s+s\beta(x_s') ) +C|\beta(x_s')-e_n| \cdot |\nabla u(x_s+s\beta(x_s'))| ]\\
		&\leq \lim_{t\to0^+}  \left(   \phi(x_s')+\sigma_{3}(Cs)+C|\beta(x_s')-e_n| \cdot |\nabla u(x_s+s\beta(x_s'))|   \right)\\
		& \leq \phi(x')+C|\beta(x')-e_n| \cdot \left\|u\right\|_{Lip(\bar \Omega)}\\
		& \leq \phi(x')+C|x'|\cdot \left\|u\right\|_{Lip(\bar \Omega)},
	\end{split}\]
	and \eqref{eq:nb value lip} follows.	
	
	Finally, if $u(x) \lesssim  |x|^{1+{\alpha_0}}$ on $ \partial \Omega$, note that $\beta$ is still oblique,  \eqref{eq:nb value lip} and a modified version of Lemma \ref{lem:lip on modules} imply that $|\nabla u(x)| \lesssim |x'|^{\alpha_0}$ on $ \partial \Omega$, and the same proof yields \eqref{eq:nb value c11}.
\end{proof}

%
%

\vspace{8pt}

\section{Good Shape Property}\label{sec:gsl}

  Let $u $ be a solution to \eqref{eq:oblique eq 1}, and assume that $\pom  \in  C^{1,\alpha}$,    $f \in C^{\alpha}( \bar \Omega)$ with $ 0<\lambda \leq f(x) \leq \Lambda$, $\phi \in C^{1,\alpha}(\bar{\Om})$,   $\beta \in C^{1,\alpha}(\partial \Omega; \R^n)$ is an oblique vector field. As in Section \ref{sec:Lipschitz} , we always assume that
   \[x_0=0, \text{ and }\Omega \subset \left\{ x_n \geq 0 \right\} \text{ satisfies \eqref{eq2.6 boundary prigin definition}-\eqref{eq2.8 boundary g definition} }\]
   and
   \[	u(0)=0,\quad  \nabla u(0)=0,\quad u\geq 0 \text{ in } \Omega,\quad  \beta(0)=e_n,\quad \phi(0)=0.\]
  Furthermore, when  $n \geq 3$, we assume that \eqref{eq:qg ch1}, i.e,  
  \begin{equation}\label{eq:qg chgsl}
  0 \leq u(x) \lesssim |x|^{2} , \quad \forall x \in \partial \Omega,
  \end{equation}
  and then $|\nabla u(x)| \lesssim |x'|$ on $ \partial \Omega$.  
  \begin{Lemma}\label{lem:gsl dn voln-1}
  	For all small $h>0$ and $x \in \R^n$, suppose the line $\ell(t):=x+ te_n$, $t \in \R$ intersects $S_h(0)$ at the points $y$ and $z$, then 
  	\[  |y_n-z_n| \cdot \operatorname{diam}S_h(x_0) \lesssim \frac{|S_h(0)|}{h^{\frac{n-2}{2}}} \lesssim h. \] 
  \end{Lemma}
\begin{proof}
	From Lemma \ref{lem:sec vol lbd} , we know that $|S_h(0)| \lesssim h^{\frac{n}{2}}$, so the lemma holds naturally for $n=2$. For $n \geq 3$, we notice that $ h^{1/2}B_{c}(0) \subset \P S_h(0)$ and so \[|y_n-z_n| \cdot \operatorname{diam}S_h(x_0) \cdot h^{\frac{n-2}{2}} \lesssim |S_h(0)| \lesssim h^{\frac{n}{2}}.\]
\end{proof}
By Lemma \ref{lem:scon 2dires}, we now have
\begin{Lemma}\label{lem:contrac}
	There exists a module $\sigma:\R^+\to \R^+ $ with $\lim_{h \to 0}\sigma(h) =0$ such that
	\[\operatorname{diam}S_h(x_0) \leq \sigma (h).\]
\end{Lemma}
   
    Starting from this Section, unless otherwise stated, we always consider the 
{\sl mixed problem:
	\begin{equation}\label{eq:mixed problem 1}
		\det  D^2 u  =f      \text{ in } \Omega,   \quad
		D_{n} u=\phi^0                 \text{ on } G:=\partial \Omega \cap B_c(0), 
	\end{equation}
	where $ \phi^0(x'):=D_nu(\G x')$ satisfies    $|\phi^0 | \leq C|x'|$ around $0$. }
The purpose of this section to show the good shape property around $0$.

\begin{Theorem}\label{lem4.2 good shape lemma}
		For all small $h>0$, we have
	\begin{align}
		\label{eq:gsl decompose}	&   c\P S_h(0) \subset   \P  G_h(0) \subset \P S_h(0),  \\
		\label{eq:gsl vol bound}	&    ch^{n/2} \leq \left|S_h(0)\right| \leq C h^{n/2} ,	  \\
		\label{eq:gsl balancing}	&     \P  G_h(0) \subset -C\P  G_h(0).
	\end{align}
\end{Theorem}
\begin{Remark}
	By noting that 
	\[	u(x)-u(y)- \nabla u(y) \cdot (x-y) \lesssim  C|x-y|^{2} +\varepsilon_0, \quad \forall x \in \partial \Omega\]
	holds at $y \in \partial \Omega$ sufficiently close to $0$ for small $\varepsilon_0>0$. We can also show the good shape properties at $y \in \partial \Omega$ sufficiently close to $0$ for $h \geq \sigma_{4}(\varepsilon_0)$. Then the boundary section $S_h(y)$ looks like the cylindrical domain
	\[  \left\{ z+t\beta(y) |\, y+t\beta(y) \in S_h(y) \text{ and } z\in \P^{y} G_h(y) \right\}\]
	with the projection $\P^{y} G_h(y)$ being balanced about $y$, where $\P^{y} $ is the projection along the $\beta(y)$ direction to the tangent plane of $\partial \Omega$ at $y$.
\end{Remark}

The right side of \eqref{eq:gsl vol bound} is the result of Lemma \ref{lem:sec vol lbd}.The remaining parts will be proven by Lemmas  \ref{lem:gsl decompose}, \ref{lem:gsl vol bound} and \ref{lem:gsl balancing}.

\begin{Lemma}\label{lem:gsl decompose}
		For all small $h>0$, we have
	\begin{equation} \label{eq:gsl 1decompose}
	  \P G_h(0)  \subset \P S_h(0) \subset \left(1+C \|\phi^0\|_{L^{\infty}}\right)\P G_h(0) \subset 
	  C\P G_h(0).
	\end{equation}
\end{Lemma}

\begin{proof}
	For any fixed unit vector $e$ such that $e\bot e_n$, we let 
	\[
	K:= \frac{\sup \left\{s |\, se\in  \P S_h(0)\right\}}{\sup \left\{s |\, te\in  \P G_h(0), \  \forall t\in (0,s)\right\}} .  
	\]
	
	To simplify, let us assume that  $e=e_1$ and that the two maxima in the above equation are achieved by points $y=(y_1,0,y_n)\in \overline{S_h(0)}$ and $z=(z_1,0,z_n) \in \overline{G_h(0)}$ respectively, where $y_1=\sup \left\{s |\, se_1\in  \P S_h(0)\right\}$ and $z_1=\sup \left\{s |\, se_1\in  \P G_h(0), \  \forall t\in (0,s)\right\}$. Then, $u(z)=u(y)=h$, $y_1 = K z_1$, and $ y_n \geq 0$ is small. Additionally, $z$ lies below the line connecting  $y$ and $0$, and we let
	\[H:=\frac{y_n}{K}-z_n \leq \frac{y_n}{K}\lesssim K^{-1}.\]
	Lemma \ref{lem:gsl dn voln-1} implies that	$Hy_1 \lesssim h$, and hence $KHz_1 \lesssim h$.  Since $y_1=Kz_1$, we see that
	\[z+He_n= K^{-1}y\in  \left\{sy,\  s\in (0,1) \right\} \cap \left\{z+te_n, \ t\in [0, \infty)]\right\}.\]
	The convexity implies that
	\[u(K^{-1}y) \leq K^{-1}  u(y) +(1-K^{-1})u(0)  =  K^{-1}h ,\]
	while \eqref{eq:nb value lip} implies that  
	\[u(z+He_n) \geq u(z)-  \|\phi^0\|_{L^{\infty}} H z_1 \geq (1-CK^{-1}\|\phi^0\|_{L^{\infty}})h.\]
	Combining these two inequalities, we obtain that $ K \leq  1+C\|\phi^0\|_{L^{\infty}}$.
\end{proof}


	Now, for any point $x \in \Omega$ near $0$ and constant $\kappa \in [0,1]$, we can take the point 
	\begin{equation}\label{eq:proj nb point}
		y_{\kappa,x}=\G \left(c\P( \kappa x)\right) \in \partial \Omega,
	\end{equation}
	 so that  $y_{\kappa,x}'= c\kappa x'$ and 
	\[  u(y_{\kappa,x})\leq u(\kappa x) \leq \kappa u(x)+(1-t)u(0) \leq \kappa u(x).  \]

\begin{Lemma}\label{lem:gsl vol bound}
	For all small $h>0$, we have
	\[\left|S_h(0)\right| \gtrsim h^{\frac{n}{2}}.\]
\end{Lemma}
	\begin{proof}
	Let $h$ be small. In contrast to our Lemma, we now assume that $\epsilon:=\frac{ \left|S_h(0)\right| }{h^{\frac{n}{2}}}   $ is very small. Let $m_{h}$ denote the mass center  of $S_h:=S_h(0)$, then $z:=y_{\frac{1}{2},m_{h}} \in S_h$.
	 Note that $0\in \P S_h$, so $\P S_h$ is balanced about $ z'$.  By John's Lemma \ref{lem2.9 John's Lemma}, we can find a transformation $\D'=\operatorname{diag}\left\{ d_1,\cdots ,d_{n-1}\right\}$ in some suitable orthogonal frame such that
	    \begin{equation}\label{eq:sec gsl vol1}
	    	\P S_h-z' \subset \D' B_C'(0) \text{ and } \det \D'  \approx  \operatorname{Vol}_{n-1} \P S_h.
	    \end{equation}
	    Now, let $d_n:=\sup \left\{t|\ z+te_n \in S_h(0) \right\}$ and $\D=\operatorname{diag}\left\{\D',d_n\right\}$, then
	    \begin{equation}\label{eq4.7 eqaution gsl prop 1 eq3}
	     \det  \D \approx \left|S_h\right|,
	    \end{equation}
	   Let  $\A x:= x-z-\ell(x')e_n$, where $\ell$ is the support function of $G$ at $z$,  then $ \D^{-1}\A {S}_1 \subset B_C^+(0)$. Consider the function
	    \[ w(x)=  \frac{u\left(\A^{-1}\D x\right)}{h}. \]
	  Then $w(0)=\frac{u\left(z\right)}{h}\leq \frac{1}{2}$, and by Lemma \ref{lem:gsl dn voln-1},  we have  on $\D^{-1}\A \left(\partial S_h \setminus G_h\right)$ that
	  \[ |D_nw(x)|= \left|\frac{ d_n \phi^0(\D' x')}{h} \right| \lesssim \frac{d_n|D'x'|}{h} \lesssim d_n\operatorname{diam} \left(\P S_h\right) \frac{|x'|}{h} \lesssim 	\frac{\left|S_h\right|}{ h^{\frac{n-2}{2}}} \frac{|x'|}{h} \lesssim \epsilon |x'|.\]
	 Therefore,
	  	\[\begin{cases}
	  	\det D^2 w = \frac{(\det \D)^2}{h^n}f(\A^{-1}\D x)\approx \epsilon^2    & \text{ in }  \D^{-1}\A S_h , \\
	  	|D_{n}w|  \lesssim \epsilon                &  \text{ on } \D^{-1}\A  G_h, \\
	  	w=1                   &  \text{ on }   \D^{-1}\A \left(\partial S_h \setminus G_h \right),
	  \end{cases}\]
    
	We now claim that for large enough  $K>0$, the convex function
	\[ \upsilon(y)=w(0)+\sum_{i=1}^{n}\frac{ y_i^2}{8nK^2}+ \frac{y_n}{4K} \]
	satisfies
	\[\begin{cases}
		\det  D^2 {w} < \det  D^2\upsilon        &  \text{ in } \D^{-1}\A S_h , \\
		D_{n} {w} < D_{n}\upsilon                     &\text{ on }\D^{-1}\A  G_h,  \\
		\upsilon<  {w} =1                   & \text{ on }  \D^{-1}\A \left(\partial S_h \setminus G_h\right).
	\end{cases}\]
	Since the oblique assumption $e_n \cdot \nu > 0$ still holds, according to Lemma \ref{lem:comp pN}, we have $\upsilon<w$, which contradicts the fact that  $\upsilon(0)= w(0)$.  This completes the proof.

	By direct computation, we can verify our claim as follows:
	\begin{itemize}
		\item in $\D^{-1}\A S_h$,  $\det  D^2 \upsilon \gtrsim  1 >C\epsilon^2 >\det  D^2w $;
		\item on $\D^{-1}\A  G_h $, $D_{n} \upsilon \gtrsim 1  > C\epsilon > D_{n}w$;
		\item on $\D^{-1}\A \left(\partial S_h \setminus G_h\right)$,   $\upsilon <   w(0)+\frac{1}{2} \leq 1 =w$.
	\end{itemize}
	\end{proof}

\begin{Lemma}\label{lem:gsl balancing}
	For all small $h>0$, we have
	\begin{equation}\label{eq4.22 equation gsl3bl eq1}
		\P  G_h(0) \subset -C\P  G_h(0).
	\end{equation}
	Therefore, $\P G_h(0)$ is balanced about $0$. 
\end{Lemma}
\begin{proof}
	As in the proof of Lemma \ref{lem:gsl vol bound}, we choose the point $m_{h}$ and set  $z=y_{\frac{1}{2},m_{h}}$. We also choose an appropriate orthogonal so that the diagonal transformation $\D'=\operatorname{diag}\left\{ d_1,\cdots ,d_{n-1}\right\}$ satisfies \eqref{eq:sec gsl vol1}, and let $\D=\operatorname{diag}\left\{\D',d_n\right\}$. For simplicity, we assume that $d:=d_1 \geq d_2 \geq \cdots \geq d_{n-1}  \gtrsim h^{\frac{1}{2}}$. Considering the function
		    \[\tilde{u}(x)=  \frac{u\left(\D x\right)}{h} \text{ and } \tilde{\phi^0}(x') =D_n u(\D x)= \frac{d_nD_n \phi^0(\D'x')}{h},\]
		we write $\tilde{S}_1 :=\tilde{S}_1^{\tilde{u}}= \D^{-1} S_h(0)$ and $\tilde{G}_1 :=\tilde{G}_1^{\tilde{u}}= \D^{-1} G_h(0)$. 
		
			Let $K$ be a large constant, and let $\delta =\frac{1}{K^3}$ and $\kappa = K^{-\frac{1}{n+1}}$ be small. If
		\[\P  G_h(0) \not\subset -\delta^{-1}\P  G_h(0).\]
		 From \eqref{eq:gsl decompose} (Lemma \ref{lem:gsl decompose}), we have
	\[\operatorname{dist}(0,\partial  \P \tilde{S}_1 ) \approx \operatorname{dist}(0,\partial  \P \tilde{G}_1)  \leq \delta.\]
	This implies that $d\gtrsim \delta^{-1} h^\frac{1}{2}$. 
	Therefore, we have on $\tilde{G}_1$ that
	\begin{equation}\label{eq4.23 equation gsl3bl eq2}
		|\tilde{\phi^0}(x')| \lesssim \frac{d_n|D'x'|}{h}	\lesssim
		\begin{cases}
			 \frac{d_nd_1|x_1|}{h} \lesssim \delta & \text{ if } n=2,  \\
			 	  \frac{d_nd_1|x_1|+d_nd_2}{h} 
			 \lesssim \frac{h^{\frac{1}{2}}}{d_2}	|x_1| +\frac{h^{\frac{1}{2}}}{d}
			 \lesssim \frac{h^{\frac{1}{2}}|x_1|}{d_2}+\delta & \text{ if } n\geq 3.
	\end{cases}
	\end{equation}
	
	Furthermore, when $n \geq 3$, we will describe the boundary values more carefully. Take a point $q=\tilde{\G} q'$ such that  $\operatorname{dist}(0, q')=\operatorname{dist}(0,\partial  \P \tilde{G}_1 )  $,  and without loss of generality, assume that $q_1 \geq 0$. We consider two cases:  $ d_2 \leq K h^{\frac{1}{2}}$ or $ d_2 \geq K h^{\frac{1}{2}}$.
	
	{\bf Case 1}:  $ d_2 \leq K h^{\frac{1}{2}}$. Then  we have
	\[ B_{cK^{-1}}''(0) \subset \P \tilde{S}_1 \subset B_C'(0).\]
	By convexity, we also have $\sup \left\{t|   t  e_1 \in \P \tilde{G}_1\right\}   \lesssim K\delta \lesssim K^{-2}$, so \eqref{eq4.23 equation gsl3bl eq2} becomes
	\[	|\tilde{\phi^0}(x')|\lesssim K^{-1} \text{ provided } x_1\geq -\frac{1}{K}.\]

	{\bf Case 2}:  $ d_2 \geq K h^{\frac{1}{2}}$.  Then \eqref{eq4.23 equation gsl3bl eq2} implies
	\[|\tilde{\phi^0}(x')| \lesssim K^{-1}.\]
	In this case, we consider a new orthogonal coordinate system with origin at zero and ${e}_1= q'/|q'|$ as the axis, while keeping the normal and tangent planes invariant. We still denote this coordinate system as  $(e_1, e_2,\cdots , e_n)$.

	Regardless of the case and dimension, we always define 
	\[E=\Omega \cap \left\{x_1\geq -\frac{1}{K}\right\} =\Omega \cap\left\{  -\frac{1}{K} \leq x_1  \leq \frac{C}{K^2}\right\}\]
	 and let 
	$m_{E}$ be the mass center of $E$. We take $z_E=y_{\kappa,m_{E}}$ and  $\A_{E} x:= x-z_E-\ell_E(x')e_n$, where $\ell_E$ is the support function of $\tilde{G}$ at $z_E$, and consider the function
	  \[ w(x):=\frac{\tilde{u}\left(\A_{E}^{-1}x\right)}{h}. \]
	  We have $w(0) \leq \kappa $, and the domain $F=\A_{E} E$ satisfies
	  \[F \subset   [-CK^{-1}, CK^{-1}] \times B_C''(0) \times [0,C\kappa^{-1}].\]
	  Let $a= \frac{1}{K}-z_E \cdot e_1$. Since $\A_{E} \left\{x_1 \geq -K^{-1}\right\}=\left\{x_1 \geq -a\right\} $, we can write $\partial F =\partial_1 F \cup \partial_2 F \cup \partial_3 F$, where
	  \[ \partial_1 F =\partial F \cap \A_{E} \tilde{G}_1(0) \cap  \left\{ y_1 > -a\right\} ,  \ \
	  \partial_2 F=\partial F \cap \left\{ y_1 = -a\right\},  \ \
	  \partial_3 F=\partial F  \setminus (\partial_1 F \cup \partial_2 F). \]  
 Since $|z_E  \cdot e_1| \leq C\kappa K^{-1} $, we have $a\approx K^{-1}$.
		Noting that $-y \cdot	p  \leq  w(0) -w(x) \leq  w(0)\leq \kappa$ holds for any $x$ on $\partial_2 F$ any any $p \in \partial w(x) $, we have the following.
	\begin{equation}\label{eq4.24 equation gsl3bl eq3}
		\begin{cases}
			\det D^2 w\approx c   & \text{ in }  F , \\
			D_{n}w \lesssim K^{-1}      &  \text{ on } \partial_1 F ,  \\
			D_{-y}w  \lesssim \kappa          &  \text{ on } \partial_2 F , \\
			w=1                   &  \text{ on }  \partial_3 F.
		\end{cases}
	\end{equation}

	Let
	\[Q(x)= \frac{1}{4n}\left(\frac{x^2}{2}+2x\right), \ x\in \R \]
	and consider the convex function
	\[ \upsilon(y)=w(0)+\left[Q\left(\frac{y_1}{ a}\right) +\kappa \left(Q\left(\frac{ \kappa y_n}{ C}\right)+  \sum_{i=2}^{n-1}Q\left(\frac{y_i}{ C}\right) \right) \right]. \]
	We claim that
	\begin{equation}\label{eq:proof gsl 4}
		\begin{cases}
			\det D^2 w < \det D^2 \upsilon   & \text{ in }  F, \\
			D_{n}w  < D_{n}\upsilon   &  \text{ on } \partial_1 F  , \\
			D_{-y}w  <  D_{-y} \upsilon        	              &  \text{ on } \partial_2 F , \\
			\upsilon < w                  &  \text{ on }  \partial_3 F.
		\end{cases}
	\end{equation}
	Then, we can apply the comparison principle\footnote{A short segment along the vector $-y$ at point $y\in \partial_1 F \cap \partial_2 F$ lies in $\bar F$, and a modified version of Lemma \ref{lem:comp pN}  still implies  $ \upsilon < w$.} to obtain $ \upsilon < w$, which contradicts the fact that  $\upsilon(0)= w(0)$.  This completes the proof.

By direct computation, we can verify \eqref{eq:proof gsl 4} as follows:  
	\begin{itemize}
		\item in $F$,  $\det  D^2 \upsilon \gtrsim\kappa^{n+1} a^{-2} \gtrsim K^2 \kappa^{n+1}  \gtrsim K  >\det  D^2 {w}$;
		\item 	on $\partial_3 F$,  $\upsilon\leq \upsilon(0)+ \frac{3n}{4n} = w(0)+ \frac{3}{4} \leq  \kappa+ \frac{3}{4} <1={w}$;
		\item on $\partial_2 F$,  $D_{-y}\upsilon \geq D_{-y_1} Q(\frac{y_1}{ a}) -C\kappa \geq c-C\kappa > \kappa \geq D_{-y} w(0)$;
		\item 	on $\partial_1 F$,  $D_{n}\upsilon \gtrsim \kappa^2 >CK^{-1}\geq D_{n}w  $.
	\end{itemize}
\end{proof}

\section{Uniformly Strict Convexity for Normalization Family}\label{sec:uscl}
Following Section \ref{sec:gsl}, we now introduce the {\sl (Sliding) Normalization Family} related to the oblique derivative problem. We can assume that $\nabla' \phi(x')=ae_1$.
Fix a small $h$ and  $y_h \in \P S_h (0)$ such that
 \[ |y_h \cdot e_1|=\sup\left\{ x\cdot e_1 |\, x\in \P S_h(0)\right\}.\]
Consider the sliding transformation
\[\A_h x=x+  \sum_{i=2}^{n-1}\frac{y_h \cdot e_i}{y_h\cdot e_1} x_1  e_i.\]
Using the good shape lemma (Theorem\ref{lem4.2 good shape lemma}),  we can find a new orthogonal coordinates in which the directions of $e_1$ and $e_n$ remain unchanged, and
\[	\D_h B_{c(n)}(0) \subset \A_h^{-1}  S_{h}(0) \subset  \D_h B_{C (n)}(0) \]
holds for some diagonal matrix $\D_h :=\operatorname{diag}(d_{1}(h), \cdots, d_n(h))$, where we further require
\[\Pi_{i=1}^n d_{i}(h)=\det \D_h=|S_h|.\]
Note that the sliding transformation $\A_h$ does not change the tangent plane $\R^{n-1}$. Assuming that $d_2 \geq d_3 \geq \cdots \geq d_{n-1} \geq ch^\frac{1}{2}$, we also have
\[	d_1 d_2 d_n \leq h^{\frac{3}{2}} \text{ when } n \geq 3.\]
\begin{Definition}\label{def4.6 sliding normalization}
	Let $\T_h=\A_h\circ\D_h= \operatorname{diag}\left\{\T_h',d_n(h)\right\}$, we define the sliding normalization $(\tilde{u}, \tilde{\Omega} )$ of $(u,\Omega)$ as follows
	\[	\tilde{u}_h(x):= \frac{u(\T_h x)}{h}, \quad x \in \tilde{\Omega}_h:= \T_h^{-1} \Omega.\]
	For simplicity, we omit the subscript $h$ when there is no confusion. $\tilde{u}$ is a solution to 
	\begin{equation} \label{eq4.16 slide normalization equation}
		\begin{cases}
			\det D^2 \tilde{u}=\tilde{f}       & \text{ in }\tilde{S}_1 , \\
			D_{n} \tilde{u} =\tilde{\phi^0 }                &  \text{ on }\tilde{G}_1 , \\
			\tilde{u}=1                   &  \text{ on }\partial  \tilde{S}_1 \setminus \tilde{G}_1 ,
		\end{cases}
	\end{equation}
	where
	\[	\tilde{f}(x) =    \frac{(\det  \D)^2}{h^n}f(\T_h x),\quad \tilde{\phi^0 }(x') =\frac{d_n\phi^0  ( \T'_h x')}{h}.\]	
\end{Definition}

  We now define the  {\sl Neumann coefficient}
\[a_h=\frac{d_1(h) d_n(h) }{h}a ,\quad  \text{ then }\left|a_h\right|\lesssim  1, \quad \text{ and if } n\geq 3, \quad  \left|a_h\right| \lesssim\frac{h^\frac{1}{2}}{d_2} \lesssim 1. \] 
We will see that there exists ${\alpha_0} >0$ such that
\begin{equation}\label{eq:normal phi c1a}
	\left|\tilde{\phi}_h(x')^0-a_hx_1\right| \lesssim  |d_n(h)|^{\alpha}  \text{ on } \tilde{G}_1,
\end{equation}
Also,  the quadratic growth assumption \eqref{eq:qg chgsl} gives, 
\begin{equation}\label{eq:qgc n-2 limit}
	|a_h|	\tilde{u}_h(\G (0,x'')) \lesssim \sum_{i=2}^{n-1}d_i(h)^2x_i^2
\end{equation}

 In this section, we discuss the uniformly strict convexity of the solutions of normalization families, which will be used in the next section to compensate for the lack of compactness. 



\begin{Theorem}[Uniformly Strict Convexity]\label{lem:uscl}
	There exists small positive constant  $\delta_0$ such that  
	\begin{align}
		\label{eq:uscl  upper}	&   (1+\delta_0)\delta_0S_{h}(0) \cap \Omega \subset S_{\delta_0h}(0) ,  \\
		\label{eq:uscl lower}	&    (1+\delta_0)S_{\delta_0 h}(0) \cap \Omega \subset   S_{h}(0) , \\
		\label{eq:uscl engulf}	&   S_{\delta_0 h}(x) \subset  c S_{h}(0)  \text{ if } x\in S_{\delta_0 h}(0).
	\end{align}
\end{Theorem}
By iteration, \eqref{eq:uscl  upper} and \eqref{eq:uscl lower} imply for $m =1,2,\cdots,$
\[ \delta_0^{m \left(\log \left( 1+\delta_0\right)+1\right)}  S_{h} (0)\subset  S_{\delta_0^m h} (0)   \subset \delta_0^{m\log (1-\delta_0)} S_{h}(0). \] 
Thus,
\[ 
c_{\delta_0}\left|x-x_0\right|^{\frac{1}{{\alpha_0}}}\leq u\left( x \right)-u(x_0)-\nabla u(x_0) \cdot \left(x-x_0\right) \leq C_{\delta_0}\left|x-x_0\right|^{1+{\alpha_0}} \text{ for } x \in S_{ t,p}^u(0),
\]
with ${\alpha_0}= \inf\left\{\log \left( 1+\delta_0\right), \frac{1}{\log (1-\delta_0)},\alpha \right\} \in (0,1)$.
Recalling Lemma \ref{lem:lip on modules}, it will imply that
\begin{equation}\label{eq:uscl gradient c1a}
	| \nabla \tilde{u}(x)| \lesssim |x|^{{\alpha_0}} \text{ in } \tilde{S}_1(0).
\end{equation}

 We refer to {\sl \eqref{eq:uscl  upper} as the upper uniformly strict convexity lemma and \eqref{eq:uscl lower} as the lower strict convexity lemma}.   By \eqref{eq:uscl gradient c1a}, the engulf property \eqref{eq:uscl engulf} is a direct consequence of \eqref{eq:uscl  upper} and \eqref{eq:uscl lower}.  
\begin{proof}[\bf  Proof of Theorem \ref{lem:uscl} in the case of  $ \bf n =2$]
	We can apply Lemma \ref{lem:contrac} to the normalization $(\tilde{u}_h,\tilde{S}_1^h)$ to obtain that $\operatorname{diam}\tilde{S}_t^h(0) \leq \sigma (t)$, which  yields \eqref{eq:uscl lower}.  Then, \eqref{eq:uscl  upper} follows from \eqref{eq:uscl lower} and  the good shape lemma \ref{lem4.2 good shape lemma}.
\end{proof}

\begin{Remark}
	When $n=2$, recall \eqref{eq:uscl gradient c1a}, or when $n \geq 3$,  recall \eqref{eq:nb value c11},  we always have $\left|\phi^0(x')-ax_1\right| \lesssim |x'|^{1+{\alpha_0}}$ on $G$ , and thus \eqref{eq:normal phi c1a}  is given by direct calculation as follows:
		\[	\left|\tilde{\phi}^0(x')-a_hx_1\right| \lesssim \frac{|d_1| \cdot|d_n|^{1+{\alpha_0}} }{h} \lesssim |d_n|^{{\alpha_0}} \lesssim |\sigma (h)|^{{\alpha_0}} \text{ on } \tilde{G}_1,\]
\end{Remark}

\begin{Remark}\label{lem5.3 local to global for normalized} 
	Let  $\Theta$ be any local estimate or property that is related to $(\tilde{u},\tilde{\Omega})$, and assume that  $\Theta$ is invariant under the normalization transformation and linear transformations that preserve the tangent plane.  From \eqref{eq:uscl engulf}, it follows that if $\Theta$ holds locally (near $0$), then $\Theta$ holds globally in $B_C(0)$. In particular, we can use Theorem \ref{lem:uscl} to obtain Lipschitz estimates for $\tilde{u}$ and  $\tilde{g}$ on $B_C(0)$ and $B_C'(0)$.	
%
\end{Remark}

The remaining part of this section is mainly devoted to the proof of Theorem \ref{lem:uscl} in the case of $n\geq 3$. This is divided into two cases: the first case where $|a_h|$ is small and the result will hold naturally, and the second case where $|a_h|$ is not small and we need to use the  quadratic growth assumption \eqref{eq:qg chgsl} and the qualitative strict convexity lemma \ref{lem:scon 2dires}.

\subsection{The Zero Neumann Boundary Value Problem}\label{sec5.1}
%
 We first consider the case where $|a_h|=0$, which leads us to study the following degenerate problem. 
In this section, we consider the reflection $R(x',x_n)=(x',-x_n)$, and the reflection image of a set $S$ is denoted by $RS:= \left\{(x, -x_n)|\, x\in S\right\}$
\begin{Theorem}\label{lem:0nbv problem}
	Assume that $E$ is a bounded convex domain symmetric about the plane $\R^{n-1}:=\left\{x_n=0\right\}$,  and let $f$ satisfies  $ 0<\lambda \leq \hat f \leq \Lambda$ and
	$\hat f \circ R = \hat f $ in $E$. Let 
	\[S :=  \left\{x\in \bar E|\, x_n \geq \hat g(x') \geq 0, x\in \P E\right\}\]
	 be a closed convex subset satisfying	$\P S = \P E$, where $\hat g$ is nonnegative.
	Assume that $v$ is a convex  solution to 
	\begin{equation} \label{eq:zero nbv eq}
		\det  D^2 v  =\hat f \chi_{ S\cup RS}      \text{ in } E,  \
		v  =0                   \text{ on } \partial E ,
	\end{equation}
	where $\chi $ is the characteristic function.
	Then, $D_n v= 0$ on $\hat G:=\partial S \setminus \partial E$,  $v$ satisfies the uniformly strict convexity Theorem \ref{lem:uscl} on $ \hat G$,
	and $v \in C_{loc}^{1,{\alpha_0}}(E)$ for small constant ${\alpha_0} ={\alpha_0}(n,\lambda ,\Lambda)$.
\end{Theorem} 
\begin{proof}

	
	
	\[ \hat\G x= 	\begin{cases}
		(x',|x_n|)   &   x\in S\cup RS , \\
		(x',\hat g(x'))         &  x\in E\setminus \left(S\cup RS\right).
	\end{cases} \]
Without loss of generality, we assume that  
$|S| \neq 0$. By convexity, we have $u \circ \hat\G \geq v $.   
For any points $y,z\in E$,  the point $\hat\G (\frac{y+z}{2})$ is contained in the simplex generated by the vertices $\hat\G y$, $\hat\G z$,  $R\hat\G y$ and $R\hat\G z$. Thus, 
	\[\hat\G \left(\frac{y+z}{2}\right)= \kappa\left(\frac{\hat\G y+\hat\G z}{2}\right)+(1-\kappa)\left(\frac{R\hat\G y+R \hat\G z}{2}\right)\]
	holds for some $\kappa\in [0,1]$. Therefore,   $v\left(\hat\G \left(\frac{y+z}{2}\right)\right)\leq \frac{v\left(\hat\G y\right)+v\left(\hat\G z\right)}{2}$,
  and hence $u \circ \hat\G$ is a convex function satisfying
  	\[ 
  \det  D^2 u \circ \hat\G  \geq \hat f \chi_{S\cup RS}=\det  D^2 {v}          \text{ in } E, 
  u \circ \hat\G =0 =v                  \text{ on } \partial E. \]
By the comparison principle, we have $v \geq  u \circ \hat\G $, so $u=u \circ \hat\G$.
 In particular, we have $D_n v=0 $ on $E\setminus \bar S$, which means that all the information of $v$ is contained in $\bar S$.

	For any fixed point $y \in \hat G$,  we now choose $p  \in \partial v(y)$ such that $p\cdot e_n\geq 0$ and let $S_h^{v} (y):=S_{h,p}^{v} (y)$. We claim that if $S_h^{v} (y)\cap \bar S$ is strictly contained in $E$, then $S_h^{v} (y)$ is of good shape  at $y$. More precisely, there exist positive constants $c$ and $C$ depending only on $n$, $\lambda$ and $\Lambda$ such that
		\begin{align}
		\label{eq:gsl0 decompose}	&  \P \left( S_h^{v}(y) \cap\hat G\right)=\P\left(S_h^{v}(y) \cap S\right) =\P S_h^{v}(y),  \\
		\label{eq:gsl0 vol bound} & ch^{\frac{n}{2}} \leq \left|S_h^{v}(y) \cap S\right| \leq C h^{\frac{n}{2}} ,	 \\
		\label{eq:gsl0 balancing}   &  \P  G_h^{v}(y) \subset -C\P  G_h^{v}(y),
	\end{align}
	where 	$ G_h^{v}(y) =S_h^{v}(y)\cap \hat G$.
 	We obtain \eqref{eq:gsl0 decompose}	 from the fact $v=v \circ \hat \G$, while the right-hand side of \eqref{eq:gsl0 vol bound} comes from Lemma \ref{lem:sec vol lbd}.
	Now let  $z' $ be the mass center  of $\P S_h(y)$,  let $z=  \hat \G (z',0)$, nd consider the sliding transformation $\A_z x= \left(x',x_n-\ell_z (x')\right)$,  where $\ell_z$ is any support function of $\hat G$ at $z$. Let $S_z = \A_z \left(S_h(y)  \cap S\right)\cup R  \A_z \left(S_h(y)\cap S\right)$,  $ E_z:= \A_z S_h(y) \cup R  \A_z S_h(y)$, and consider the function
	\[   v_z(x)=   v\left(\A_z^{-1}(x',|x_n|)\right)-  v(y)- p' \cdot (x-y) -h, \quad x\in E_z. \]
	  We can normalize $v_z$ and $E_z$ such that $B_{c}(z) \subset  E_z \subset B_{C}(z)$ and $\left\| \det D^2 v_z \right\|_{L^{\infty}} \approx 1$. Noting that $\P S_z=\P E_z $ and  $f \chi_{E_z}$ satisfies the double measure property at $z$ (see  \cite{[C3]} or Section 3.1 in  \cite{[G]} for the definition), we have 
	  \[-v_z(y) \geq - v_z(z) \approx 1.\]
	   Then, the Aleksandrov’s Maximum Principle implies that $E_z$ is balanced about $z$.  This gives \eqref{eq:gsl0 vol bound}  and \eqref{eq:gsl0 balancing}.
	
	By following the standard proof of Theorem 1 in   \cite{[C1]}, we can see that the good shape lemma will prohibit the existence of lines on the graph of $v$ with endpoints on $\hat{G} \cup R\hat{G}$, thereby giving the corresponding uniform strict convexity and $C^{1,\alpha}$ regularity of the solution.
	
\end{proof}
%

\vspace{8pt}


\subsection{Proof of  Theorem \ref{lem:uscl}}\label{sec5.2}

\begin{proof}


%
		By the good shape lemma, we have the decomposition
	\[c\left(\P \tilde{G}_1(0) +(\tilde{S}_1(0) \cap \left\{ te_n|\, t\geq 0 \right\})\right)\subset \tilde{S}_1(0)
	\subset C\left(\P \tilde{G}_1(0) +(\tilde{S}_1(0) \cap \left\{ te_n|\, t\geq 0\right\})\right).\]
	Using a similar iterative method, the proof of Theorem \ref{lem:uscl} is equivalent to showing separate proofs along the normal direction $\tilde{S}_1(0) \cap \left\{ te_n: t\geq 0\right\}$ and the tangent direction $\P \tilde{G}_1(0)$.
	
	We first   prove \eqref{eq:uscl  upper}
	in the normal direction. 	Let $k=\frac{1}{2C}<1$, it is enough to prove
	\[\delta\tilde S_1(0)\cap \left\{te_n|\, t\geq 0\right\}\subset k \tilde S_{\delta}(0)\cap \left\{te_n|\, t\geq 0\right\} \text{ for a small } \delta >0 .\]
	Otherwise, we have
	\[ \tilde{u}(te_n) \geq kt \text{ for } t\geq \delta.\]
	According to Lemma \ref{lem2.5 balance lemma} and good shape Lemma (Theorem \ref{lem4.2 good shape lemma}),
	\[	\frac{\operatorname{Vol}_{n-1}  \left(\P\tilde{S}_t(0) \cap \P (-\tilde{S}_t(0)) \cap B_c'(0) \right)}{t^{\frac{n-2}{2}}}  \geq ck \text{ for } t\geq c\delta.\]
	Since $\left\|\tilde{u}\right\|_{L^{\infty}\tilde{S}_1(0)} \leq 1$, Lemma \ref{lem:quasi sconx} implies $\delta \geq \sigma_{1}(k)$, which is a contradiction.
	\vskip 0.5cm
	
		 Let $\kappa<\sigma_{1}(k) $ be a small positive constant, and let ${\alpha_0}>0$ be a constant that is smaller than the values given by Theorem \ref{lem:0nbv problem}, taking
	\[ K=\kappa^{-n-1}, \delta =K^{-\frac{2(n+7)(1+{\alpha_0})}{{\alpha_0}}} \text{ and } \delta_0= \delta^{2n+6}.\] 
	To prove \eqref{eq:uscl  upper} in the tangent direction and \eqref{eq:uscl lower}, we assume that $h$ is small enough such that $ \operatorname{diam}S_{h_0}(0) \leq \delta^4$ and consider two cases.

	\textbf{Case 1}. There exists $s \in [\delta^2 h,h]$  such that $d_2(s)\geq Ks^{\frac{1}{2}}$.
	For simplicity, let us assume that $s=h$ and $ d_2(h) \geq Kh^{\frac{1}{2}}$.
	In this case, \eqref{eq:normal phi c1a} implies $\left|\tilde{\phi^0}\right| \leq  \left|\frac{Ch^{\frac{1}{2}}}{d_2}\right|+C\sigma(h)^{\alpha} \lesssim  K^{-1}$ on $B_{C}'(0)$.  Thus, we can use Lemma \ref{lem:gsl decompose} for $\tilde{u}$ to obtain 
	\[\P \tilde{S}_1(0)  \subset (1+CK^{-1})\P \tilde{G}_1(0).\]
	In  Theorem \ref{lem:0nbv problem}, we take $\hat f=\tilde{f}$, $S=\tilde S_1$. Let $E$ be the  convex hull of $S \cup RS$, and let  $v$ be the solution to 
	\[	\det  D^2  v  =\hat f \chi_{ S\cup RS}      \text{ in } E,  \
	 v  =0                   \text{ on } \partial E . \] 
	 Combined with the Aleksandrov’s Maximum Principle, we have
	\[  v(x) \geq  -C(n,\Lambda) \operatorname{dist}(x,\partial E)^{\frac{1}{n}} \geq - CK^{-\frac{1}{n}} \geq -\kappa, \quad  \forall x
\in  \partial \tilde{S}_1 \setminus \partial \tilde{G}_1. \]
	By comparing with the functions $w^{\pm}(x):=\frac{1+v(x)}{1\mp \kappa}\pm2\kappa(2C-x_n)$, we obtain that
	\[ \left\|\tilde u-1-v\right\|_{L^{\infty}} \lesssim \kappa.\]
Since $\tilde u(0)=0$,   we have
	\[\begin{split}
		 \min\left\{\tilde u(\G x'), \tilde u(\G (-x'))\right\}  
		 & \leq  \min\left\{v(\G x'), v(\G (-x'))\right\}-1 \\ 
		 & \leq  \min\left\{v(\G x'), v(\G (-x'))\right\}-v(0)+ C\kappa \\
		  &\leq C|x'|^{1+{\alpha_0}}+ C\kappa.
	\end{split}\]
	and
	\[\begin{split}
		\max\left\{\tilde u(\G x'), \tilde u(\G (-x'))\right\}   
		& \geq  \max\left\{v(\G x'), v(\G (-x'))\right\} \\ 
		& \geq  \max\left\{v(\G x'), v(\G (-x'))\right\}-v(0)- C\kappa  \\
		& \geq  c|x'|^{\frac{1+{\alpha_0}}{{\alpha_0}}}- C\kappa.
	\end{split}\] 
	Recalling \eqref{eq:gsl0 balancing}, we obtain
	\[ c|x'|^{\frac{1+{\alpha_0}}{{\alpha_0}}}- C\kappa \leq  \tilde u(\G x') \leq C |x'|^{1+{\alpha_0}}+ C\kappa ,\]
	which  proves both \eqref{eq:uscl  upper} and \eqref{eq:uscl lower} in the tangent direction.

	Next, we prove   \eqref{eq:uscl lower} in the normal direction. That is
	\[\tilde S_{\delta_0}(0)\cap\left\{te_n|\, t\geq 0\right\}\subset (1-\delta_0)\tilde S_1(0).\]
	Suppose for contradiction that there exists a point $y\in \tilde S_{\delta_0}(0)\setminus (1-\delta_0)\tilde S_1(0)$ such that $y\cdot e_n>1-\delta_0$ and
	$u(y)\leq \delta_0$. Let $z=(1-\delta_0)e_n$. By convexity, we have
	$ \tilde u(z) \leq \delta_0. $
	
	Without loss of generality, assume that  $\tilde{u}(e_n)=1$. By noting that
		\[ \tilde{u}(x)\geq c|x'|^{\frac{1+{\alpha_0}}{{\alpha_0}}}- C\kappa \geq c|x_n|^{\frac{1+{\alpha_0}}{{\alpha_0}}}- C\kappa\geq 4\delta_0 \left(x_n-\frac{1}{2}\right)\text{ on } \partial \tilde{S}_1  \cap \left\{x_n \geq \frac{1}{2}\right\}.\]
	we see that $F:=\left\{ x\in \tilde S_1|\tilde{u}(x) \leq 3\delta_0 \left(x_n-\frac{1}{2}\right)\right\}$ satisfies $F \subset \subset \tilde S_1$. Consider the convex function $w(x)=\tilde u(x)-3\delta_0 \left(x_n-\frac{1}{2}\right)$, we have  $  w(z) \approx - \delta_0 \approx \inf_F w  $. Since  $\lambda \le \det D^2w\le \Lambda$ in $F$ and $ w=0 $ on   $\partial F$, $F$ should be balanced about $z$. However, this is impossible since $z=(1-\delta_0)e_n$, but $\frac{7}{8}e_n \in F$ and $e_n \notin F$, which can be inferred from $w\left(\frac{7}{8}e_n\right)<0$ and $w( e_n)>0$.
	
	\textbf{Case 2}. We have $d_2(s) \leq K s^{\frac{1}{2}}$ for all $s \in [\delta^2 h,h]$. Note that $d_2(s) \geq \cdots \geq d_{n-1}(s) \geq cs^{\frac{1}{2}}$,  which implies that
	\begin{equation}\label{eq5.13 n-2 matrix estimate}
		cs^{\frac{1}{2}}\I \leq  \D_{s}'' \leq CKs^{\frac{1}{2}}\I ,\quad \forall s \in [\delta^2 h,h].
	\end{equation}
	Hence, we have
	\begin{equation}\label{eq5.14 normal slide expansion}
		C\sum_{i=2}^{n-1} K^{-2}x_i^2\leq  \tilde{u}(\G (0,x'')) \leq C\sum_{i=2}^{n-1} K^2x_i^2 \text{ for } |x''| \geq CK\delta.
	\end{equation}
	Since \eqref{eq:uscl  upper} holds along the normal direction, given $t \geq \delta^2$, we still have
	\[ \tilde{u}(te_n) \lesssim t^{1+{\alpha_0} },\quad  \text{i.e.,}\quad 
	  t^{\frac{1}{1+{\alpha_0} }}d_n(h) \lesssim 	d_n(th).\]
	\eqref{eq5.13 n-2 matrix estimate} also shows
	\[ \Pi_{i=2}^{n-1} d_i(th) \geq  K^{\frac{2-n}{2}}t^{\frac{n-2}{2}} \cdot \Pi_{i=2}^{n-1} d_i(h) . \]
	Recalling \eqref{eq:gsl vol bound}, we have
	\[	d_1(th)\lesssim K^{\frac{n-2}{2}}t^{\frac{{\alpha_0}}{1+{\alpha_0} }}d_1(h) .\]
	This means that 
	\begin{equation}\label{eq5.15 d1 decay}
		\P	\tilde  S_t \subset \P \tilde S_1 \cap \left\{ |x_1| \leq CK^{\frac{n-2}{2}}t^\frac{{\alpha_0}}{1+{\alpha_0}}\right\} .	
	\end{equation}
	Now, choose a point $y_t \in  \P \tilde S_{t}$ such that $y_t \cdot    e_1\gtrsim \frac{d_1(th)}{d_1(h)}$, then we have
	\begin{equation} \label{eq5.16 repsent point on e1}
		t \lesssim y_t \cdot    e_1  \lesssim K^{\frac{n-2}{2}}t^{\frac{{\alpha_0}}{1+{\alpha_0} }}.
	\end{equation}
	
	First, we make the following observation. Given points  $P \in \P \tilde{S}_t$ and $Q \notin \P \tilde{S}_t$, the balancing property of $\P \tilde{S}_t$ implies that
	$\pm cP\in \P \tilde{S}_t$ and $\pm CQ \notin \P \tilde{S}_t$. From  \eqref{eq5.14 normal slide expansion}, we obtain
	\[ ct^{\frac{1}{2}}B_{1}''(0) \subset \P \tilde{S}_t \cap \R^{n-2} \subset Ct^{\frac{1}{2}}B_{1}''(0). \]
	Consider the cones
	\[ \Gamma_1^{\pm}(P):=\left\{ \mp cP+s( x \pm cP)|\,  s \geq 0 ,\ x \in ct^{\frac{1}{2}}B_{1}''(0)  \right\} \]
	and
	\[ \Gamma_2^{\pm}(Q):=\R^{n-1} \setminus \left\{ \pm CQ-s( x \mp CQ)|\, s\geq 0,\ x \in Ct^{\frac{1}{2}}B_{1}''(0)  \right\}.\]
	Let
	\[\Gamma_1(P)  = \Gamma_1^+(P)  \cap \Gamma_1^{-} (P) \text{ and } \Gamma_2(Q) = \Gamma_2^+(Q)  \cap \Gamma_2^{-} (Q) \]
	By considering rays starting from $\pm cP\in \P \tilde{S}_t$ passing through the ball  $ Ct^{\frac{1}{2}}B_{1}''(0)$, we notice that  $\P \tilde{S}_t \cap \left\{ \pm x_n \geq 0\right\}$ is contained in the cone $ \Gamma_1^{\pm}(P)$, hence $\P \tilde{S}_t \subset \Gamma_1(P)$.
	Similarly, by considering rays starting from  $\pm CQ\notin \P \tilde{S}_t$ passing through the ball  $ct^{\frac{1}{2}}B_{1}''(0)$,  we also have $	\P \tilde{S}_t \subset \Gamma_2(Q)$.
	 We now divide the proof into 3 steps. 
	
	\textbf{Step 1}.
	We prove  \eqref{eq:uscl lower} in the tangent direction. Let $\bar{t}  = t^{\frac{{\alpha_0}}{1+{\alpha_0}}}  $, where $t$ is very small such that $CK^{\frac{n-1}{2}}t^{\frac{{\alpha_0}}{2(1+{\alpha_0})}} \leq \frac{1}{2}$.  From \eqref{eq5.16 repsent point on e1},  we have
	\begin{equation}\label{eq5.17 section slide chara low}
		\P \tilde S_{\bar{t} }(0) \subset \Gamma_1(y_t) \subset \left\{ |x_1| \geq c(\bar{t} /K)^{\frac{1}{2}} (|x''|-CK^{\frac{1}{2}}\bar{t} ^{\frac{1}{2}})\right\}.
	\end{equation}
	Since $\bar{t}  \geq t$,  \eqref{eq5.15 d1 decay}  and \eqref{eq5.17 section slide chara low} imply
	\begin{align*}
		\P \tilde  S_{t}(0)
		& \subset  \P \tilde  S_{\bar{t} }(0)  \cap \P \tilde  S_{t}(0)  \\
		& \subset \left\{ |x_1|  \geq c(\bar{t} /K)^{\frac{1}{2}} (|x''|-CK^{\frac{1}{2}}\bar{t} ^{\frac{1}{2}})\right\}  \cap \left\{ |x_1| \lesssim K^{\frac{n-2}{2}} t^{\frac{{\alpha_0}}{1+{\alpha_0}}}\right\} \\
		& \subset \left\{|x''| \lesssim K^{\frac{1}{2}}\bar{t} ^{\frac{1}{2}} + K^{\frac{n-1}{2}}\bar{t} ^{-\frac{1}{2}}t^{\frac{{\alpha_0}}{1+{\alpha_0}}}\right\} 
		\cap \left\{ |x_1| \lesssim K^{\frac{n-2}{2}} t^{\frac{{\alpha_0}}{1+{\alpha_0}}}\right\} \\
		& \subset \left\{|x''| \lesssim K^{\frac{n-1}{2}}\bar{t} ^{\frac{1}{2}} \right\}
		\cap \left\{ |x_1| \lesssim K^{\frac{n-2}{2}} \bar{t} \right\}  \\
		& \subset \left\{|x'| \lesssim K^{\frac{n-1}{2}}\bar{t} ^{\frac{1}{2}} \right\} \\
		& \subset  CK^{\frac{n-1}{2}}\bar{t} ^{\frac{1}{2}} \P \tilde{S}_1=CK^{\frac{n-1}{2}}t^{\frac{{\alpha_0}}{2(1+{\alpha_0})}}\P \tilde{S}_1,
	\end{align*}
 the proof is complete.
	
	\textbf{Step 2}.
	 The proof of \eqref{eq:uscl lower} in the normal direction is the same as in Case 1, where we only used \eqref{eq:uscl lower} in the tangent direction.

	\textbf{Step 3}. Finally,  we prove the upper strict convexity lemma
	\eqref{eq:uscl  upper} in the tangent direction.   By iteration, \eqref{eq:uscl lower} implies 
	\[	\tilde{u}(x)\geq c|x|^{\frac{1+{\alpha_0}}{{\alpha_0}}} \text{ on } \tilde  S_{1}.\]
	Using similar discussions as in \eqref{eq5.14 normal slide expansion} and\eqref{eq5.15 d1 decay},  we have for any $t \geq \delta^2 $ that
	\[	d_n(th) \leq  ct^{\frac{{\alpha_0} }{1+{\alpha_0} }}d_n(h) ,\quad 
	 	d_1(th) \geq CK^{\frac{2-n}{2}}t^{\frac{1}{1+{\alpha_0} }}d_1(h).\]
	In particular, we can find $y_t\in \P \tilde S_{t}(0) $ such that
	\begin{equation}\label{eq5.18 d1 up bound slide}
		y_{t} \cdot e_1  \geq c \frac{d_1(th) }{d_1(h) } \geq cK^{\frac{2-n}{2}} t^{\frac{1}{1+{\alpha_0} }}.
	\end{equation}

	Letting
	$	b_1(t)=\sup\left\{ \mu|\, \mu e_1 \in \P \tilde S_{t}(0)  \right\}$, we assert that 
	\begin{equation}\label{eq5.19 d1 up bound claim}
		b_1(\delta_0) \geq  K\delta_0 .
	\end{equation}
This would imply $\tilde u(te_1)\leq K^{-1} t$  for $t \leq \delta_0$, and combining with \eqref{eq5.14 normal slide expansion},  we complete the proof.

	 Suppose \eqref{eq5.19 d1 up bound claim} is not true, 
	since $\frac{b(t)}{t}$  monotonically decreases and has a positive lower bound, we have     $ct \leq  b_1(t) \leq Kt $  for all $t \geq \delta_0$, implying
	\[	\P \tilde S_{t}(0)  \subset \Gamma_2(Ke_1) \subset \left\{ |x_1| \leq CK^{\frac{3}{2}} t^{\frac{1}{2}} |x''|+CKt\right\}. \]
	Recalling \eqref{eq5.17 section slide chara low}, we obtain
	\[	\P \tilde S_{t}(0) \subset \left\{  c  (t/K)^{\frac{1}{2}} |x''|-Ct
	\leq |x_1|\leq CK^{\frac{3}{2}} t^{\frac{1}{2}} |x''|+CKt\right\}.\]
	Let
	$N=CK^{\frac{3}{2}}$, $r =\frac{1}{2N^4}$,   $t=\delta^{\frac{3}{2}}  $ and $s=rt \geq \delta^2$. For any point   $y\in \P \tilde S_s \subset  \P \tilde  S_t$, we have
	\[ c  (t/K)^{\frac{1}{2}} |y''|-Ct
	\leq |y_1|\leq CK K^{\frac{1}{2}}s^{\frac{1}{2}} |y''|+CKs \leq Ns^{\frac{1}{2}} |y''|+CKs ,\]
	which implies
	\[ |y''| \leq CK\frac{t+s}{(t/K)^{\frac{1}{2}}-Ns^{\frac{1}{2}}} =CKK^{\frac{1}{2}}t^{\frac{1}{2}} \frac{1+r}{1-N(Kr)^{\frac{1}{2}}} \leq  Nt^{\frac{1}{2}} \frac{1+r}{1-N^2r^{\frac{1}{2}}},\]
	 and thus
	\[	|y_1| \leq Ns^{\frac{1}{2}}|y''|+CKs \leq N(rt)^{\frac{1}{2}}Nt^{\frac{1}{2}} \frac{1+r}{1-N^2r^{\frac{1}{2}}}+Nrt \leq Nt \left( r+Nr^{\frac{1}{2}}\frac{1+r}{1-N^2r^{\frac{1}{2}}}\right).\]
	Now taking  $y=y_s$  and revisiting  \eqref{eq5.18 d1 up bound slide}, we get
	\[	 cK^{\frac{2-n}{2}} (rt)^{\frac{1}{1+{\alpha_0} }}\leq   y_s \cdot e_1  \leq Nt \left( r+Nr^{\frac{1}{2}}\frac{1+r}{1-N^2r^{\frac{1}{2}}}\right).\]
	Therefore,
	\[\begin{split}
		\delta^{\frac{3}{2}} =t  
		&\gtrsim \left[NK^{\frac{n-2}{2}}   r^{-\frac{1}{1+{\alpha_0} }} \left( r+Nr^{\frac{1}{2}}\frac{1+r}{1-N^2r^{\frac{1}{2}}}\right)\right]^{-\frac{1+{\alpha_0} }{{\alpha_0} }} \\ & \gtrsim \left[NK^{\frac{n-2}{2}}   N^{\frac{4}{1+{\alpha_0} }} N^{-1} \right]^{-\frac{1+{\alpha_0} }{{\alpha_0} }}
		\\ &
		\gtrsim  \left[K^{\frac{n-2}{2}}   K^{\frac{6}{2(1+{\alpha_0} )}}  \right]^{-\frac{1+{\alpha_0} }{{\alpha_0} }} 
		\\ &
		\gtrsim \delta,
	\end{split} \]
	which is impossible. 
\end{proof}

\vspace{8pt}

\vspace{8pt}

\section{Viscosity Subsolutions and the Compactness of Oblique Boundary Values.}\label{sec:viscosity subsolution}
 Equation \eqref{eq4.16 slide normalization equation} has a well-behaved form and Theorem \ref{lem:uscl} provides compactness of solutions. However, $n \geq 3$, due to the lack of a uniform $C^{1}$ modules, we cannot guarantee the compactness of the corresponding oblique boundary values. Therefore, we employ the viscosity subsolution.
 
The assumptions in this section are independent of the assumptions in our main theorem. We always assume that $f$  is bounded and positive. We will use  $USC(E) $ (or $LSC(E)$) to denote the family of all upper (lower) semicontinuous functions on the set $E$. 

Recalling the oblique derivative problem
\begin{equation}\label{eq:oblique eq ch7}
	\det  D^2u  =f(x)      \text{ in } \Om  , \quad D_{\beta}u  = \phi(x)                 \text{ on } \pom.
\end{equation} 
 Caffarelli \cite{[G]} established the equivalence between the definition of a generalized solution and a viscosity solution of $\det D^2 u=f$ when $f$ is continuous 
 \footnote{Also, a generalized solution is a  viscosity subsolution (or supersolution) to the following equation
 	\[\det  D^2 u(y) \geq  \varliminf_{x\to y}f(x) ( \leq \varlimsup_{x\to y}f(x) )\].}.
    Assuming that  $\det D^2 u=f$, we say that $u \in USC(\bar \Omega)$ is a subsolution to $D_{\beta}u = \phi $ on $ \partial \Omega$, if for any $x_0\in \partial \Omega$ and any convex $v\in C^1(\bar \Omega)$ such that $u-v$ has a local maximum at  $x_0 $, we have
	\begin{equation*}
		D_{\beta}v(x_0)\geq \phi(x_0).
	\end{equation*}
We define supersolutions in a similar way. \footnote{Since a convex function on $\bar \Omega$ is always upper semicontinuous, supersolutions are actually continuous.}
	If $u\in C(\bar \Omega)$,  then the definition of viscosity subsolution to
$D_{\beta}u \geq \phi $ is weaker than the definition in terms of Dini derivatives. On the other hand, we have

%

\begin{Lemma}\label{lem3.7 viscosity subsolution lemma}
	Let $u \in  L^{\infty}(\bar  \Omega) \cap  USC (\bar  \Omega) $  be a subsolution of \eqref{eq:oblique eq ch7}. Assume that  $\beta \in Lip(\partial \Omega)$ and $\phi \in LSC(\partial \Omega)$.
	Then,  we have $D_{\beta} u\geq \phi $ on $\partial \Omega$ in the Dini sense and $u \in Lip(\bar \Omega)$.
	
\end{Lemma}
\begin{proof}
	We can assume by contradiction that  $D_{\beta} u(x_0)< \phi(x_0) $ at $x_0 \in \partial \Omega$. Without loss of generality, we let  
	 $x_0=0$, $\Omega$  satisfies \eqref{eq2.6 boundary prigin definition}-\eqref{eq2.8 boundary g definition}, $u(0)=0$,  and $\beta(0)=e_n$.	Let ${b}=\phi(0) \in \R$, and there exist small positive constants $\tau$ and $ \epsilon$ such that $D_n u(\tau e_n) \leq {b}-\epsilon$. Therefore, we have
	\begin{equation}\label{eq3.13 sub contrary}
		u(te_n)  \leq u(0)+({b}-\epsilon) t \text{ for }  t\in [0,\tau].
	\end{equation}
	
Let $\eta$ be the oblique constant,  $\kappa=\frac{1}{8(1+\left\|\beta\right\|_{Lip})}$ and $r= r(\epsilon , \tau, u ) \leq \min\left\{\epsilon, \eta_0\eta_1, \frac{1}{8}\kappa\eta_1^3\epsilon \tau\right\}$ sufficiently small such that,
	\[\phi(x)\geq {b}-\frac{\epsilon}{2} \text{ on } B_{2\eta^{-1}r}(0) \cap \partial\Omega ,\quad  u(x)\leq  \frac{1}{8}\kappa\eta^3\epsilon \tau \text{ in } B_{2\eta^{-1}r}(0) \cap \Omega ,\]
	and
	\[ u(x)\leq u( {\tau} e_n)+ \frac{1}{8}\kappa\eta^3\epsilon \tau \text{ in } B_{2r}( {\tau} e_n)  \cap \Omega.\]
	Consider the cylindrical domain
	\[\Gamma=( B_r'(0) \times [0,\tau]) \cap \Omega\]
	and the convex function
	\[ v (x)= \frac{ \kappa\eta \epsilon }{r^2}|x'|^2 +  (  {b} -\epsilon ) x_n .\]
	Let $G_1= \overline{\partial \Gamma \setminus \partial \Omega} $ and $G_2= \partial \Gamma \cap \partial \Omega$.  It is clear that  $ G_2 \subset ( B_r'(0) \times [0, \tau]) \cap \partial \Omega$ and $\partial G_1 \subset  B_{2\eta^{-1}r}(0) \setminus B_r(0)$.
	By calculation, we have
	\[  v > u\text{ on }  \partial G_1 \cup B_r'( {\tau} e_n).\]
	Note that $v$ is linear along the $e_n$ direction and $u$ is a convex function, so we have $v>u$ on  $G_1$.  For points $z  \in G_2$, we also have
	\begin{align*}
		D_{\beta(z)}v(z)  & \leq   {b}- \epsilon +\left( \frac{ 2\kappa\eta \epsilon}{ r} + 1 \right)  |\beta (z)-\beta (0)|  \\
		& \leq \phi (z) -\frac{1}{2}\epsilon +\left(\frac{ 2\kappa\eta \epsilon}{ r} + 1  \right) \left\|\beta\right\|_{Lip}r  <  \phi (z) = D_{\beta(z)}u(z).
	\end{align*}  
	In summary,
	\[	\begin{cases}
		\det  D^2v =0\leq   \det  D^2u      & \text{ in } \Gamma , \\
		u <  v                   &\text{ on }  G_1 , \\
		D_{\beta}v <  D_{\beta}u                     & \text{ on } G_2.
	\end{cases}\]
	By the comparison principle, we find that $u < v$ in $\overline{\Gamma}$,	which contradicts  $v(0)=0=u(0)$. Therefore, we have proved that $D_{\beta}u \geq \phi $ on $\partial \Omega$.
	
	Using the same discussion as in the proof of Lemma \ref{lem:lipschitz bound dC0}, we have
	\[ |\nabla u(x_0+t\beta(x_0))| \lesssim ( \omega_{u}(\Omega)   +\left\|\min\left\{\phi,0\right\}\right\|_{L^{\infty}(\partial \Omega)} ),\quad  \forall  x_0 \in \pom \text{ and } t>0 \text{ small }.\]
	Now, we consider the Lipschitz extension of $u$, $u^*(x) = \varlimsup_{ \Omega \ni y\to x}u(y)$.
 Since, $u\in USC (\bar \Omega)$ is convex, $u^* \leq u$ on $\bar \Omega$ and $u=u^*$ in $\Omega$. If $u^* < u$ at some boundary point $x_0$, say $x_0=0$, we still have \eqref{eq3.13 sub contrary}, and the same discussion shows that this is impossible. Therefore, we have  $u=u^*\in Lip(\bar \Omega)$.
\end{proof}

Let  $\Omega$  satisfies \eqref{eq2.6 boundary prigin definition}-\eqref{eq2.8 boundary g definition}, we consider the following mixed problem,
\begin{equation}\label{eq:mixed problem 2}
		\det  D^2 u  =f      \text{ in } \Omega,   \quad
	D_{n} u=\phi                 \text{ on } G:=\partial \Omega \cap B_c(0).
\end{equation} 
For our purposes, we make some additional assumptions to give the following existence and compactness results,  one of which is a qualitative strict convexity assumption on $u$, where for any fixed ${\alpha_0} >0 $, we assume that for  $\kappa_0= \kappa_0({\alpha_0})>0$ that
\begin{equation}\label{eq:visc com qlsc 1}
	c|x|^{\frac{1}{1+{\alpha_0}} } -\kappa_0 \leq u(x) \leq C|x|^{1+{\alpha_0} }+\kappa_0 .
\end{equation}

\begin{Lemma}\label{lem:visc exist mix}
	Suppose $u^-  \in USC(\bar{ \Omega}) $ is a subsolution and $u^+  \in LSC(\bar{ \Omega}) $ is a supersolution to \eqref{eq:mixed problem 1}, with  $u^- \leq u^+$ and both satisfying \eqref{eq:visc com qlsc 1}.  If $\phi= \phi(x') \in C(B_c'(0))$ is concave, then there exists a convex function $u \in Lip (\bar \Omega\cap B_C(0))$ such that 
	\[u^-\leq u\leq u^+ \text{ on } \bar \Omega\cap B_C(0)\]
	 and 
	\[\det  D^2 u  =f      \text{ in } \Omega,   \quad
	D_{n} u=\phi                 \text{ on } G:=\partial \Omega \cap B_{\rho}(0),\] 
	where $\rho>0$ is a small universal constant.
\end{Lemma}
\begin{proof}
	
	 Consider the non-empty set
	\[ V:=\left\{v \in USC(\bar \Omega) |\: v \text{ is a subsolution to problem  \eqref{eq:mixed problem 1}   and }  u^-\leq v \leq u^+ \right\}.\]
	Then the function
	\[ u(x)=\sup_{v\in V} v(x) \in Lip\left( \bar{ \Omega} \cap \bar B_c(0)\right)\]
	 is locally bounded and convex in $\Omega$, and the classical interior discussion shows that   $\det D^2u=f $ in $\Omega$, as seen in  \cite[Section 9]{[Cr]}. By Lemma \ref{lem:lipschitz bound dC0} and Lemma \ref{lem3.7 viscosity subsolution lemma}, functions in $V$ are uniformly Lipschitz on $ \bar{ \Omega} \cap B_c(0)$, and for any point  $x_0 \in \partial \Omega \cap  B_c(0) $,  we can find a sequence $\left\{v_k\right\}\subset V$ such that $v_k(x_0) \to u(x_0)$, then the supporting planes of $v_k $ at $x_0$ subconverge to the supporting plane of $u$ at $x_0$. Thus, we have  $D_{n} u \geq \phi $ on $ \partial \Omega \cap  B_c(0)$.
	
	It remains to be shown that  $D_n u\le \phi$  on $  B_{\rho}(0) \cap G$. We assume by contradiction that  $ D_n u(x_0) \geq  \phi (x_0)+3\epsilon$ for some  $x_0 \in  G \cap B_{\rho}(0)$ and   $\epsilon>0$.
	Let 
	\[E:=\left\{u(x)= u(x_0)+\nabla u(x_0) \cdot (x-x_0) \right\} \cap \bar \Omega.\] 
	From \eqref{eq:visc com qlsc 1},  we have    $E \subset \subset B_c(0) $. Since $\det D^2 u \approx 1$,  $E$ cannot have any interior extreme points in $\Omega$. Therefore,
	\[ \P E =\P \left(E\cap G\right), \text{ and } \partial E \cap \left(\partial (\P E) \times \R\right) = \partial E \cap G.\]
	Recall that $D_n u \in USC(\bar \Omega\cap B_c(0))$, $D_nu\geq \phi$ on $G$,  and $\phi$ is concave,
	Thus, we can choose $y  \in E \cap G$ such that $y'$ is an exposed point\footnote{An exposed point of a convex set $E$ is a point $x\in E$ at which some continuous linear functional attains its strict maximum over $E$, the set of exposed points is a non-empty subset of the set of extreme points.} of $\P E$ and satisfies
	$ D_n u(y) \geq  \phi (y)+3\epsilon$.
	Without loss of generality, we assume 
	\[y=0,\quad \nabla u(0)=0,\quad  \P E \subset \left\{x_1 \leq 0\right\} \text{ and } \P E \cap \left\{x_1 = 0\right\}=\left\{0\right\}. \]
	
	According to the definition of supersolutions, we have $u(0) < u^+(0)$, observing that $u^+-u \in LSC $, we can choose  positive $\tau$ and $r$, where $r<<\epsilon$, such that
	\[u^+ -u \geq \tau >0 \text{ in } B_r(0)\cap \Omega\]
	and
	\[ \phi (x)  \leq \phi(0)+\epsilon\leq -2\epsilon \text{ in } B_r(0)\cap\partial \Omega. \]
	
	Let $S_h(0): =S_{h,\nabla u(0)}^u(0)$, choose a small constant  $h >0$ (to be determined later), and let
	\[t_h:=\sup\left\{t|\, te_n \in S_h(0)  \right\} = \frac{h}{\varepsilon_h} \]
	where $\varepsilon_h \to 0$ as $h\to 0$. By Lemma \ref{lem:sec vol lbd}, we have $\left|S_h(0)\right|\leq Ch^{\frac{n}{2}}$. So Lemma \ref{lem2.5 balance lemma} implies that
	\[\operatorname{Vol}_{n-1} {\P S_h(0)} \lesssim \frac{\left|S_h(0)\right|}{t_h}  \lesssim \varepsilon_h h^{\frac{n-2}{2}}.\]
	By John's Lemma \ref{lem2.9 John's Lemma}, there exists point $x_h \in \P S_h(0)$ and an affine transformation $\T'$ on $\R^{n-1}$ such that
	\[  \P S_h(0) -x_h \subset \left\{ x'\in \R ^{n-1}|\, |\T' x'|^2 \leq c\right\} \text{ and } \operatorname{Vol}_{n-1}\left(\P S_h(0)\right)\det  \T' \approx 1.  \]
	Note that $u \in Lip(\bar \Omega) $  and $u(0)=0$ ensure that $	cB_h'(0) \subset \P S_h(0)$ and $h\left\|\T'\right\| \leq c$,
	and together with $0 \in \P S_h(0)$, we obtain
	$\P S_h(0) \subset \left\{x'\in \R ^{n-1}|\, |\T' x'|^2 \leq 1\right\}$ for small $h>0$.
	
	Let
	\[v=u+\epsilon x_n, \quad    E_h=(\left\{x'\in \R ^{n-1}|\, |\T' x'|^2 \leq 1\right\} \times [0,\epsilon^{-1}h]) \cap \bar \Omega,\]
	we have
	\begin{equation}\label{eq:visc domain Sv}
		\hat S_h^v(0):=\left\{x\in \Omega |\; v(x)<h\right\} \subset \left\{u \leq h\right\} \cap \left\{ x_n\leq \epsilon^{-1}h \right\} \subset  E_h.
	\end{equation}
	Let  $a_1(h) :=\sup\left\{t |\, \G (te_1 ) \in S_h(0) \right\}$ and consider the set
	\[ F_h =	\hat S_h^v(0) \cap \left\{ -16a_1(h) \leq  x_1 \leq a_1(h)\right\}.\]
		Since $\lim_{h \to 0 }\left(S_h(0) \cap \left\{ -16a_1(h) \leq  x_1 \leq a_1(h)\right\}\right) =E\cap \left\{x_n=0\right\}=\left\{0\right\}$ in the Hausdorff sense, and $\P E \subset \left\{x_1 \leq 0\right\}$, we have  $\lim_{h \to 0 } F_h \to \left\{0\right\}$.  We now choose $h>0$ small such that $F_h \subset B_r(0)$. Let
		\[  Q^h(x)= P^h(x) +\frac{h x_1}{8a_1(h)},\]
		where
	\begin{equation}\label{eq:visc function P}
		 P^h(x)=\frac{h}{4}+\frac{ h}{2n}\left[|\T' x'|^2+\left(\frac{x_n}{ \epsilon^{-1}h} \right)^2\right ]-\epsilon x_n .
	\end{equation}
Then,
	\[ Q^h(x) \leq P^h(x)+\frac{h}{8} <h = u \text{ on } \left\{\partial E_h\setminus G\right\} \cap \left\{ x_1 \geq -8a_1(h)\right\}\]
	and thus
	\[ Q^h(x) \leq P^h(x)-2h <-h < u \text{ on } \left\{\partial F_h\setminus G\right\} \cap \left\{ x_1 = -8a_1(h)\right\}.\]
	Noting that  $\det D^2 Q^h=\det D^2 P^h \geq C \epsilon^{2}|\det   \T'|^2h^{n-2}  \geq\frac{ C\epsilon^2 }{\varepsilon_h^2} >> 2\Lambda$, we obtain
	\[\begin{cases}
		\det D^2 Q^h  >2\Lambda &
		\text{ in } F_h  , \\
		D_{n}Q^h > \phi+\frac{\epsilon}{2}  &
		\text{ on } \partial F_h \cap G , \\
		Q^h(x) <u-\frac{h}{8}  &
		\text{ on } \partial F_h \setminus G .
	\end{cases}\]
	In addition, since the function $w=\max\left\{Q^h, u\right\} \chi_{F_h} +u\chi_{F_h^c} $ satisfies $w \leq u^+$, so we have $w \in V$. However, $w(0)=Q^h(0)> u(0)$, which contradicts the definition of $u$.

\end{proof}

 Similarly, we have compactness results for the following mixed problems.

\begin{Lemma}\label{lem:visc comp mix}
	Let $u_k  $, $k=1,2,\cdots$, satisfy
	\[	 \det  D^2u_k=f_k     \text{ in } \Om_k\cap B_c(0),  \
	D_{n}u_k  = \phi_k                    \text{ on } \pom_k \cap B_c(0), \]
	where $\Om_k=\left\{ x|\, x_n \geq g_k(x'), \ x'\in B'_c(0) \right\}$  satisfies \eqref{eq2.6 boundary prigin definition}-\eqref{eq2.8 boundary g definition}, $ \lambda \leq  f_k \leq \Lambda$, and $\phi_k \in L^{\infty}(B_c(0))$ converges uniformly to a concave function $ \phi \in C\left(B_c(0)\right)$. Suppose all $u_k$ satisfy \eqref{eq:visc com qlsc 1}. Then $\left(u_k,\Omega_k\right)$ subconverges to some  $(u,\Omega)$, and $u$ satisfies
	\[		
	\lambda \leq \det  D^2u   \leq \Lambda      \text{ in } \Om\cap B_{\rho}(0),  \
	D_{n}u  = \phi                    \text{ on } \pom\cap B_{\rho}(0).\]
\end{Lemma}

\begin{proof}
	By Theorem \ref{lem:lipschitz bound id},  $u_k$is locally uniformly Lipschitz, hence 
	$u_k$ subconverges to the subsolution  $u$ of some mixed problem
	\[  \det  D^2u =f ,\quad \lambda \leq f \leq \Lambda       \text{ in } \Omega,  \text{ and } D_\beta u \geq \phi   \text{ on } \partial \Omega. \]
	 The same proof as in Lemma \ref{lem:visc exist mix} can also verify that this is an upper solution.
\end{proof}

Furthermore, we state the following existence and compactness theorems, although they will not be used in the proof.
\begin{Theorem}\label{thm:exist theorem}
	Assuming that $f$  is bounded and positive,  and $\beta \in Lip ( \partial \Omega; \R^n)$  is oblique, and  $\phi(x,r) \in C(\bar \Omega \times R)$ with $D_r\phi(x,r)\geq b>0$ for some constant $b$. Suppose $n=2$ or $\Omega$ is a strictly convex domain, then the Robin problem 
	\begin{equation} \label{eq1.4 robin problem}
		\det  D^2u  =f(x)      \text{ in } \Om,  \quad
		D_{\beta}u  = \phi(x,u)                   \text{ on } \pom.
	\end{equation}
	has a solution $u \in Lip (\bar{ \Omega})$.
\end{Theorem}

\begin{proof}[\bf Proof of Theorem \ref{thm:exist theorem}]
	By choosing appropriate positive constants $K_1 $ and $K_2$, we have $ u^+(x) = K_1$ as a supersolution, and $ u^-(x) = -K_2+\Lambda|x-y|^2$  as a subsolution of problem\eqref{eq1.4 robin problem}. By the same proof as in Lemma \ref{lem:visc exist mix}, we only need to show that the function $u(x)=\sup_{v\in V} v(x) $ satisfies $D_{\beta}u  =  \phi(x,u)    u$, where 
	\[ V:=\left\{v \in USC(\bar \Omega) |\, v \text{ is a subsolution to problem \eqref{eq1.4 robin problem}, }  u^-\leq v \leq u^+\right\}.\]
	
	Without loss of generality,  we assume by contradiction that $0\in \partial \Omega$, $\Omega$ satisfies \eqref{eq2.6 boundary prigin definition}-\eqref{eq2.8 boundary g definition}, $\beta(0)=e_n$, $u(0)=0$, $\nabla u(0)=0$,  $ D_n u(0) \geq  \phi \left(0,0\right)+3\epsilon$. Then, we can proceed with the same discussion and notation as in the proof of the theorem, the only difference being that the proof is simpler in this case. Here, we only point out where the differences lie.   We choose $\hat{S}_h^v$ as in \eqref{eq:visc domain Sv} and  $P_h$ as in \eqref{eq:visc function P}.	If $n=2$,  then by Lemma Lemma \ref{lem:scon 2dires} and the fact  $v \geq u$, we have 
 	$ \hat{S}_h^v(0) \subset  S_h(0)  \to \left\{0\right\} $  as $h \to 0 $. If $\Omega$ is strict convex,  by \eqref{eq:visc domain Sv} we have
 	$\hat{S}_h^v(0) \subset  \left\{ x_n\leq \epsilon^{-1}h\right\} \to \left\{0\right\} $   as  $ h \to 0 $.
 	By writing $\T' x'=\sum_{i=1}^{n-1}a_i^{-1}x_ie_i$  in suitable orthogonal coordinates,  we obtain \[\left| D_{x'}\left(\frac{ h|\T' x'|^2 }{2n}\right)\right| \lesssim \sum_{i=1}^{n-1}ha_i^{-2}x_i \lesssim \sum_{i=1}^{n-1}ha_i^{-1} \lesssim 1 ,\]
 	and then
 	\[ D_{\beta}P^h \geq {\beta'}(x')\cdot D_{x'}\left(\frac{ h|\T' x'|^2 }{2n}\right)-C|\beta_n(x')-\beta_n(0)|-\epsilon > -2\epsilon  \text{ on }  \partial E_h \cap G. \] 
 	Therefore, we still have
		\[	\begin{cases}
			\det D^2 P^h  >2\Lambda &
			\text{ in } E_h  , \\
			D_{\beta}P^h > \phi  &
			\text{ on } \partial E_h \cap G , \\
			P^h(x) <u  &
			\text{ on } \partial E_h \setminus G .
		\end{cases}\]
	In addition, since the function  $\hat w=\max\left\{P^h, u\right\} \chi_{E_h} +u\chi_{E_h^c} $ satisfies $\hat w \leq u^+$, so we have $\hat w \in V$. However, $\hat w(0) =P^h(0)> u(0)$, which contradicts the definition of $u$. 
\end{proof}
\vspace{8pt}

Similarly, we state the following compactness results without proof.
\begin{Theorem}\label{lem3.13 compactness theorem}
		Let $u_k  $, $k=1,2,\cdots$, satisfy
		\[	 \det  D^2u_k=f_k     \text{ in } \Om_k,  \
	D_{\beta_k}u_k  = \phi_k                    \text{ on } \pom_k ,\]
	where $\Om_k=\left\{ x|\, x_n \geq g_k(x'), \ x'\in B'_c(0) \right\}$  satisfies \eqref{eq2.6 boundary prigin definition}-\eqref{eq2.8 boundary g definition}, $ \lambda \leq  f_k \leq \Lambda$, $\beta_k $ is uniformly oblique on $\partial \Omega_{k}$ for $k$,  and $\phi_k \in L^{\infty}(B_c(0))$ converges uniformly to a concave function $ \phi \in C\left(B_c(0)\right)$. Suppose $n=2$ or $\partial \Omega$ is strictly convex. Then, up to a constant, $u_k$ subconverges to a solution $u$ of 
		\[	 \det  D^2u=f     \text{ in } \Om,  \
	D_{\beta}u  = \phi                    \text{ on } \pom. \]
	
\end{Theorem}

%



%

\section{Stationary Lemma and \texorpdfstring{$C^{2,\alpha}$}{c2 alpha} Regularity}\label{sec:liouville and station}

%

 	Following Section \ref{sec:uscl}, by the virtue of \eqref{eq4.16 slide normalization equation}- \eqref{eq:normal phi c1a},  Theorem \ref{lem:uscl} and Lemma \ref{lem:visc comp mix} imply that the (Sliding) normalization of $u$ and the corresponding oblique derivative problem are pre-compact. And as $h \to 0$, we can ensure convergence in arbitrarily large domain.  Suppose $\hat u$ is one of the limits, and assume for simplicity that
 	\[\begin{cases}
 				\operatorname{det}D^2 \hat{u}=1     &  \text{ in } \hat{S}_1,\\
 				D_{n} \hat{u} =ax_1             &  \text{ on } \hat {G}_1, \\
 				\hat{u}=1                   &  \text{ on }  \partial \hat{S}_1 \setminus \hat{G}_1.
 			\end{cases}\]
 	
 
\begin{Lemma}\label{lem:c11 normal 2}
We have	$ \hat u(te_n) \lesssim  t^2.$
\end{Lemma}
\begin{proof}
	If $n=2$,  	If $\hat u \in C_{loc}^{1}( \hat \Omega\cup \hat G_1)$ and strictly convex in $ \hat \Omega$, which is always true when $n=2$, we have that $\hat u$ is smooth in $ \hat \Omega$. Therefore, the function  
	$\zeta=D_n \hat u-ax_1 $ is continuous in $\hat \Omega\cup \hat G_1$ and satisfies
	\[	\begin{cases}
		\hat U^{ij}D_{ij}  \zeta =0  &  \text{ in }  \hat S_1(0),  \\
		\zeta \lesssim 1     &           \text{ on }  \partial \hat S_1(0), \\
		\zeta = 0      &           \text{ on } \hat G_1(0),
	\end{cases}\]
	where $\hat U^{ij}$ is the cofactor matrix of $D^2\hat u$.  The Lemma is completed by using  
	\[w^+(x)=C_1\left[\hat u-\frac{n }{2} x_nD_n\hat u +C_2nx_n\right]\]
	 as an  upper barrier, where  $C_1$ and $ C_2$ are sufficiently large.

	 If $n\geq 3$, for any fixed $\epsilon>0$ small, we now approximate $\hat S_1(0)$ from the inside using a smooth convex domain $E$, and take smooth functions $f_{\epsilon}$ and $w_{\epsilon}$ such that $D_n f_{\epsilon}=0$ and
	\[\left\|f_{\epsilon}-1\right\|_{L^{1}(E)}+ \left\|w_{\epsilon}-\hat u\right\|_{L^{\infty}(\partial E)} \text{ is arbitrary small }.\]
	By solving the Dirichlet problem
	\[	\det  D^2u_{\epsilon}  =f_{\epsilon}     \text{ in } E,  \quad
	u_{\epsilon} =w_{\epsilon}                 \text{ on } \partial E,\]
	we can always find a convex function such that
	\[\left\|u_{\epsilon} -u\right\|_{L^{\infty}(E)} \leq \epsilon \]
	We further assume that $ S^\epsilon:=\hat S_{\frac{1}{2}}(0) +\epsilon^{\frac{1}{2}} e_n \subset E $ and let $G^\epsilon= \hat  G_{\frac{1}{2}}(0)+\epsilon^{\frac{1}{2}} e_n$.

	By convexity, we have $D_n u_{\epsilon} \lesssim 1$ in $S^\epsilon$.
	Note that there exists  $\sigma_{3}$  (depending on $\hat u$) with $\sigma_{3}(0)=0$ such that
	\[  0 \leq  D_n \hat u\left(\hat \G x'+ te_n\right) -D_n \hat u(\hat\G x') \leq \sigma_{3} (t).\]
	Therefore, for any point $q \in S^\epsilon$, we have by convexity that
	\[  D_n u_{\epsilon}(p) \leq  \frac{\hat u\left(p+\epsilon^{\frac{1}{2}}e_n\right)- \hat u(p)+2\epsilon}{\epsilon^{\frac{1}{2}}} \leq \ell(p')+\sigma_{3}(2\epsilon^{\frac{1}{2}})+\epsilon^{\frac{1}{2}} ,  \]
	and
	\[  D_n u_{\epsilon}(p) \geq  \frac{\hat u\left(p-\epsilon^{\frac{1}{2}}e_n\right)-\hat u(p)-2\epsilon}{\epsilon^{\frac{1}{2}}} \geq \ell(p')-\sigma_{3}(\epsilon^{\frac{1}{2}})-\epsilon^{\frac{1}{2}}. \]
	And Lemma \ref{lem:lip on modules} implies that locally $\left\| \nabla u_\epsilon\right\| \leq C$.
	Thus, $\zeta_\epsilon=D_nu_\epsilon-\ell(x) $ is a bounded solution of  
	\begin{equation}\label{eq:lthm linearized Dn 1}
		\begin{cases}
			U_\epsilon^{ij}D_{ij}  \zeta_{\epsilon}  =0  &  \text{ in }  S^\epsilon,  \\
			\zeta_\epsilon \leq C      &           \text{ on }  \partial S^\epsilon, \\
			\zeta_\epsilon \leq \sigma(2\epsilon^{\frac{1}{2}})+\epsilon^{\frac{1}{2}}       &           \text{ on } G^\epsilon.
		\end{cases}
	\end{equation}

	According to Theorem \ref{lem:uscl}, there exists a universal small constant  $\varrho >0$  such that  
\begin{equation}\label{eq6.1 quantative strictly convex def}
	\tilde S_{\frac{1}{2}}(0)\subset \frac{1}{8{\varrho}} \tilde S_{ {\varrho}}(0),
\end{equation}
holds for for each normalized solution $\tilde{u}$. This property remains invariant under uniform convergence. Therefore, it also holds for $\hat{u}$. Let 
\[\bar u_\epsilon(x)=u_\epsilon(x)-u_\epsilon(y_{\epsilon})-\nabla \hat u(y_{\epsilon}) \cdot (x-y_{\epsilon}),\]
where $y_{\epsilon}=\epsilon^{\frac{1}{2}} e_n$, we have
	\[\bar u_\epsilon \geq c >0    \text{ on } \partial S^\epsilon \setminus G^\epsilon.\]
	In fact, we notice that the function $\ell(x)=u_\epsilon(y_{\epsilon})+\nabla \hat u(y_{\epsilon}) \cdot (x-y_{\epsilon})-\varrho -2\epsilon$ satisfies
	\[ \ell(0) \geq -\varrho-c_1 \varepsilon^{\frac{1}{2}}\epsilon,\quad \ell  \leq 0 \text{ on } \hat S_{\varrho}(0).\]
Looking back at \eqref{eq6.1 quantative strictly convex def}, we find that
	\[  \ell+{\varrho} \leq \frac{1}{4} \text{ on } \partial \hat S_{\frac{1}{2}} \setminus \hat G_{\frac{1}{2}}.\]
	This means that $u-\ell-{\varrho} \geq \frac{1}{4}$ on $ \partial \hat S_{\frac{1}{2}} \setminus \hat G_{\frac{1}{2}}$, and hence
	\[\bar u_\epsilon= u_\epsilon-\ell-{\varrho}-2\epsilon \geq \frac{1}{8}\text{ on } \partial S^\epsilon \setminus G^\epsilon.\]
 
 The Lemma is completed by using 
		\[w_{\epsilon}^+:=C_1\left[\bar u_\epsilon-\frac{n }{2}(x_n-\epsilon^{\frac{1}{2}} )D_nu_\epsilon +nC(x_n-\epsilon^{\frac{1}{2}} )\right] +\sigma(2\epsilon^{\frac{1}{2}})+\epsilon^{\frac{1}{2}}  , \]
	as supersolution for the equation of $D_n u_{\epsilon}$ and then letting taking $\epsilon \to 0$.
%
\end{proof}

By Lemma \ref{lem:c11 normal 2}, there exists $\delta_1>0$ and module $\sigma^*$ such that when $\omega_{f}(\Omega)  +\omega_{D\phi}(\Omega) +h \leq \delta_1$, 
\begin{equation}\label{eq:c11 normal almost}
\tilde{u}(te_n) \leq Ct^2+\sigma^*(\delta_1) \text{ for } h \leq h_0.
\end{equation} 
This will ensure that the Neumann boundary $\tilde{G}$ converges locally uniformly to the plane  $\R^{n-1}$, and by verifying the assumptions in Liouville Theorem, the limit ${\hat u}$ must be quadratic, giving us

\begin{Lemma} \label{lem:sta c0 perturb}
	 For any small constant  
	 $\epsilon_0 >0$, we can find  constant $\delta_1 >0$ such that if	$ \omega_{f}(\Omega)  +\omega_{D \phi}(\Omega)  +h\leq \delta_1$, then there exists a quadratic function $P_h$ satisfying
	\[	P_h(0)=\nabla' P_h(0)=0,\quad \det D^2P_h=f(0),\]
	and
	\[	\left\|D_nP_h-D_n \tilde{u}\right\|_{L^{\infty}(\tilde{S}_1)}+\left\|\tilde{u}-P_h\right\|_{L^{\infty}(\tilde{S}_1)} \leq \epsilon_0.\]
\end{Lemma}
 \begin{proof}[\bf  Proof of Lemma \ref{lem:sta c0 perturb}]
 	Recalling \eqref{eq:c11 normal almost}, by iteration, this means that for any fixed $\epsilon>0$, we have
 	\[	d_n(h)\geq c_{\epsilon}h^{\frac{1+\epsilon}{2}},\]
 	 and thus we obtain $ \T_h' \leq C_{\epsilon}h^{\frac{1-\epsilon}{2}} \I'$.  By taking  $\epsilon = \frac{\alpha}{8}$, we find that
 	\[	|\tilde {g}_h(x')| \leq \frac{1}{d_n(h)}|\T_h'x'|^{1+\alpha} \lesssim h^{\frac{\alpha}{4}}|x'|^{\alpha} \to 0 ,\quad \text{ as }  h\to 0^+,\]
 	so the Neumann boundary $\tilde{G}_h$ converges locally uniformly to the plane  $\R^{n-1}$
 	
 	By recalling\eqref{eq:normal phi c1a} and \eqref{eq:qgc n-2 limit}, and letting $\delta_1 \to 0$ (so that $h\to 0$),	for any sequence of normalized  solutions, we can assume that one of the following cases holds for a subsequence: either $a_h \to 0$, thus $D_n{\tilde u} \to 0$; or $|a_h| \gtrsim K^{-1}$ for some constant $K>0$ depending on the sequence, and we obtain $\tilde u(0,x'',0) \lesssim K |x''|^2$. By applying  \ref{thm:liouville thm}, the limit ${\hat u}$ must be quadratic. 
 	
 	Note that the sections with height $1$ has already been normalized, so the limit functions are pre-compact. Therefore, for any $\varepsilon>0$, when $\delta_1$ is small, there exists a quadratic function $P_h$ satisfying
 	\[	P_h(0)=\nabla' P_h(0)=0,\quad \det D^2P_h=f(0),\]
 	and
 	\[	\left\|D_nP_h-D_n \tilde{u}\right\|_{L^{\infty}(\tilde G_1)}+\left\|\tilde{u}-P_h\right\|_{L^{\infty}(\tilde{S}_1)} \leq \varepsilon^2 .\]
 	By convexity, we have
 	\[D_n \tilde{u}\left(x',\tilde{g}(x')\right) \geq D_n P_h\left(x',\tilde{g}(x')\right) -\varepsilon^2\]
 	and 
 	\[D_n  P_h(x )-C\varepsilon \leq 	D_n \tilde{u}\left(x',\tilde{g}(x')\right)  \leq D_n  P_h(x )+C\varepsilon \text{ if } x_n \geq \tilde{g}(x')+\varepsilon,\]
 	which means that
 	\[\left\|D_nP_h-D_n \tilde{u}\right\|_{L^{\infty}(\tilde{S}_1)} \lesssim \varepsilon,\]
 	and then the proof is completed by choosing $\varepsilon =c\epsilon_0$.
 \end{proof}

Next, we use the standard perturbation method to prove that if $\tilde{u}_h$ is approximated by a quadratic function in $\tilde{S}_1^{\tilde{u}_h}$, and the known data $\tilde f$ and $\tilde \phi ^0$ are smaller perturbations of constant and linear functions, then $\tilde{u}_{\mu h}$ has  a better approximation in $\tilde{S}_1^{\tilde{u}_{\mu h}}$, where $\mu$ might be very small, ultimately giving Theorem  \ref{thm: c2alpha n=3}.

 Let $\F_{a}$ denote the set of quadratic functions given by $P(x):=Q(x)+\ell(x')$, where $\ell$ is a linear function and the quadratic term $Q(x)= \sum_{i, j=1}^na_{ij}x_ix_j$ satisfies  $D_n Q (x', 0)= ax_1$ and  $\det D^2 Q=1$.
Given $Q+\ell \in \F_a$,  let  $(\gamma_2,  \cdots , \gamma_{n-1})=(D_{1 2}Q, \cdots,  D_{1 n-1}Q)(D_{x''}^2Q)^{-1}$. By considering the sliding transformation
$\B x =x-2 \sum_{i=2}^{n-1} \gamma_ix_1e_i $ and then a rotation transformation in $x''$, 
we can always find a new coordinate system in which $Q$ takes the form of
\[	Q_{a,\kappa} (x):= \frac{1}{2}\left[\sqrt{\kappa^{n-2}+a^2}(x_1^2+x_n^2)+2ax_1x_n+\kappa^{-1}\sum_{i=2}^{n-1}x_i^2\right],\]
where $a\in R$ and $\kappa>0$. Clearly, $\operatorname{det}D^2 Q_{a,\kappa} (x) \equiv 1 $ and $D_n Q_{a,\kappa} (x', 0)= ax_1$.

In Theorem \ref{thm:stationary thm}, Lemma \ref{lem:stationary lem} and Lemma \ref{lem:station perturb 1} below, we will abuse the notation $u$, which refer to different functions in each lemma.
\begin{Theorem} \label{thm:stationary thm}
	Let  $u\in C(\overline{S_1})$ is be a solution to 
	\begin{equation} \label{eq:stationary model 1}
		\begin{cases}
			\operatorname{det}D^2 {u}=1     &  \text{ in } {S}_1,\\
			D_{n} {u} =a x_1               &  \text{ on }  {G}_1, \\
			{u}=1                   &  \text{ on }  \partial {S}_1 \setminus {G}_1.
		\end{cases}
	\end{equation}
	Suppose $	u(0)=0$, $\nabla u(0)=0$, $u\geq 0$, $|a| \leq C$,  $B_{c}^+(0)\cap \overline{{S}_1} \subset {S}_1 \subset B_{C}(0)$, and the defining function $g$ of $G_1$ satisfies $\left\|g\right\|_{Lip(B_{c}'(0))}  \leq C$. 
	Then there exists a universal constant $\epsilon_0>0$ such that if
	\[\left\|g\right\|_{C^{1,\alpha}(B_{C}'(0))} +	\left\|u-P\right\|_{L^{\infty}(S_1(0))}+  \left\|D_nu-D_nP\right\|_{L^{\infty}(S_1(0) )} \leq \epsilon_0 \]
	holds for some  $P(x)=\ell(x') +Q(x) \in \F_{a}$ with $B_{c/4}(0) \subset \left\{P \leq 1 \right\}$, then we have 
	\[\left\|u\right\|_{C^{2,\alpha}\left(\overline{B_{c}^+(0)\cap S_{1/2} } \right)} \leq C.\]
\end{Theorem}


Theorem \ref{thm:stationary thm} relies on the following lemma. 
\begin{Lemma} \label{lem:stationary lem}
		Let  $u\in C(\overline{S_1})$ is be a solution to  \eqref{eq:stationary model 1}. 
Suppose $P(x)=\ell(x') +Q(x) \in \F_{a}$ satisfies $B_{c/4}(0) \subset \left\{P \leq 1 \right\}$  and
	\begin{equation} \label{eq7.2 small coefficients 1}
		\left\|g\right\|_{C^{1,\alpha}(B_{C}'(0))}+\left\|u-P\right\|_{L^{\infty}\left(S_1  \right)} + \left\|D_nu-D_nP\right\|_{L^{\infty}\left(S_1 \right)}  \leq  \epsilon,
		\end{equation}
	then  
	\begin{equation} \label{eq:l estimates 1}
	|D\ell|\leq C\epsilon^{\frac{1}{2}}. 
	\end{equation}
	For any small constant $\mu>0$, there exists  $\epsilon_0>0 $, depending on $n$, $|a|$, $\mu$ such that if $\epsilon \leq \epsilon_0 $, then we can find a better approximation $P^0=l^0+ Q^0 \in \F_{a}$ of $u$ so that
	\begin{equation} \label{eq7.4 better approximation 1}
		|u-P^0| \leq C \epsilon  \mu^{\frac{3}{2}},\quad   |D_nu-D_nQ^0|\leq  C\epsilon  \mu \text{ in } S_{\mu}^{Q^0},
	\end{equation}
	and 
	\begin{equation} \label{eq7.5 better approximation Q}
	 	\left\| D^2 Q^0 - D^2 Q \right\| \leq 2\epsilon.
	\end{equation}
\end{Lemma}
The proof of Lemma \ref{lem:stationary lem} is the same as \cite[Lemma 6.2]{[JT]}, we put it in the Appendix \ref{app:proof b}.

\begin{proof}[\bf Proof of Theorem \ref{thm:stationary thm}]  
	 Let $h_k=\mu^k$ for $k=1,2,\cdots$. By using Lemma \ref{lem:stationary lem}, we will prove the following by induction on $k$.  For simplicity, we omit all subscripts $u_k$ for domains and boundaries. 
	 	 
	{\bf Claim.}\;
	For $k\geq 0$, there exists a sequence of constants $\kappa_k \in (c,C)$ and transformations
	\[\M_k=\operatorname{diag}\left\{ \M_k',\M_{k,nn}\right\}, \quad \det  {\M_k} =1,\]
	such that at height $h_k$,  the  normalization solution $(u_{k}, \Omega_{k} )$ of $(u, \Omega)$  (see Definition \ref{def4.6 sliding normalization}) given by
	\[	u_{k}(x):= \frac{u(\T_k h_k^{\frac{\alpha}{2}} x )}{h_{k}}, \text{ for }  x \in  \Omega_k:=  \T_k^{-1}  (\Omega) \text{ with } \T_k=\Pi_{i=0}^{k} \M_k\]
	satisfies
	\[	\epsilon_{k}:=\left\|g_k\right\|_{C^{1,\alpha}(B_{C}'(0))}+\left\|u_k-P_k\right\|_{L^{\infty}(S_1)} + \left\|D_nu_k-D_nQ_{a, \kappa_k}\right\|_{L^{\infty}(S_1)}  \leq  \epsilon_0\]
	for some $P_k=\ell_k+Q_{a, \kappa_k} \in \F_{a}$, where $g_k$ is the defining function of the oblique boundary of $u_k$, and 
	\[	Q_{a,\kappa} (x):= \frac{1}{2}\left[\sqrt{\kappa^{n-2}+a^2}(x_1^2+x_n^2)+2ax_1x_n+\kappa^{-1}\sum_{i=2}^{n-1}x_i^2\right] .\]
	Moreover,
	\[\left\|{\M}_k-\I\right\|+|\kappa_k-\kappa_{k-1}|\leq C\epsilon_{k-1}  \text{ and } |D \ell_k| \leq C\epsilon_{k-1}^{\frac{1}{2}}. \]
	
	The statement for $k=0$ is trivial by taking $\M=\I$, we also assume for simplicity that $\kappa_0=1$. Assuming the statement holds at $k$, we will prove the statement at $k+1$ by applying Theorem \ref{lem:stationary lem} to $u_k$. We will estimate $\left\|\T_k \right\|$ and $\epsilon_{k}$.  Assuming  $\epsilon_0 $ and $\mu$ are sufficiently small, once we prove the statement, we will have 
	\begin{equation}\label{eq7.14 estimate epsilon}
		\epsilon_{i} \leq h_{i}^{\frac{{\alpha}}{8}} \text{ for } i=0,1,\cdots,k
	\end{equation}
	Then we have $ \left\|\M_i -\I\right\|\leq 2\epsilon_{i} \leq  2 h_{i}^{\frac{{\alpha}}{8}}$ for $i=0,1,\cdots,k$, and
	\[ \left\|\T_{k}-\I\right\| \leq   \sum_{i=0}^{k}\left\|\M_i -\I\right\| \leq  4\sum_{i=0}^{k}h_{i}^{\frac{{\alpha}}{8}} \leq \frac{1}{2}  \text{ and then }  \frac{1}{2} \leq \kappa_{k} \leq 2.\] 
	We then apply Lemma \ref{lem:stationary lem} to $u_k$ and obtain a better approximation 	$\tilde{P}_{k+1}=\tilde{l} +\tilde Q  \in \F_{a}$. By \eqref{eq7.5 better approximation Q}, we can find  $\kappa_{k+1} >0$  and  positive definite  matrix $\M_{k+1} =\operatorname{diag}\left\{ \M_{k+1}',\M_{k+1,nn}\right\}$ such that
	\[D^2 \tilde Q  = \M_{k+1}D^2Q_{a,\kappa_{k+1}} \M_{k+1}^T  \]
	and
	\[\left|\kappa_{k+1}-\kappa_{k} \right|+ \left\|\M_{k+1}-\I \right\| \leq C{\epsilon}_k. \]
	 Let  $\ell_{k+1}=\tilde{l}\circ\M_{k+1}$. By calculation, we still have
	\[\left\|\T_{k+1} -\I\right\| \leq \frac{1}{2},\quad \frac{1}{2} \leq \kappa_{k+1} \leq 2\]
	and 
	\[ \left\|g_{k+1}\right\|_{C^{1,\alpha}(B_{C}'(0))} \leq Ch_k^{\frac{{\alpha}}{2}} \left\|\T_{k+1} \right\|^{n+\frac{{\alpha}}{2}} \leq Ch_k^{\frac{{\alpha}}{2}} \left\|\T_{k} \right\|^{n+\frac{{\alpha}}{2}}\left\|\M_{k+1} \right\|  \leq  h_{k+1}^{\frac{{\alpha}}{8}}. \] 
	Combined with \eqref{eq7.4 better approximation 1}, we obtain
	\[  \epsilon_{k+1} \leq  C \epsilon_{k} \mu^{\frac{1}{2}} +\left\|g_{k+1}\right\|_{C^{1,\alpha}(B_{C}'(0))}  \leq   h_{k+1}^{\frac{{\alpha}}{8}}.\]
	Therefore, we obtain \eqref{eq7.14 estimate epsilon} and the claim is complete. 
	
	Here, since we also obtained universal bounds for  
	 $\kappa_{k} $ and $\left\|\T_k\right\|$,  the same discussion allows us to refine the estimate  \eqref{eq7.14 estimate epsilon} to  $\epsilon_{k} \leq C_{\alpha}h_{k}^{\frac{{\alpha}}{2}}$.
	 Now, we can let $\kappa_{k}$ and $\T_k$ converge geometrically to $\kappa_{\infty} $ and $\T_{\infty}$, respectively, i.e., 
	\[ 	|\kappa_{k}-\kappa_{\infty}| \leq C_{\alpha}h_k^{\frac{{\alpha}}{2}}  \text{ and }  \left\|\T_{\infty}\T_{k}^{-1}-\I\right\| \leq C_{\alpha}h_k^{\frac{{\alpha}}{2}}.\]
	Replacing each $\T_{k}$ with $\T_{\infty}$, we see that 
	\[	|u(x ) -h_k^{\frac{1}{2}} \ell_k(\T_{\infty}^{-1} x)-Q_{a,\kappa_{\infty}}(\T_{\infty}^{-1}x )|  \leq C_{\alpha}h_k^{1+\frac{{\alpha}}{2}} \text{ in } B_{ch_k^{\frac{1}{2}}}(0).\]
	Recalling that $|D\ell_k|\leq C_{\alpha}\epsilon_{k-1}^{\frac{1}{2}}$, we conclude that
	\[|u(x ) -Q_{a,\kappa_{\infty}}(\T_{\infty}^{-1}x )| \leq  C_{\alpha}|x|^{2+\alpha}.\]
	Thus, we obtain the pointwise $C^{2,\alpha}$ module of $u$ at $0$.

	Note that the above discussion is valid for every point $x_0 \in G_{\frac{h_0}{2}}(0)$.
	Therefore, we find that  $u$ is $C^{2,\alpha}$ on  the boundary $ G_{\frac{h_0}{2}}(0)$.
	Assuming that  ${t} >0$  is sufficiently small, for any given point  ${z}\in \partial \Omega$ and point ${y} ={z}+ {t} \nu ( {z}) \in \Omega $ close to the origin. Let $Q_{z}$ denote the approximate quadratic polynomial of $u$ at ${z}$. 
	At this point, the section $S_{c_1{t}^2}^{Q_{z}}( {y}) $ is contained in $\Omega$ and satisfies
	\[  B_{c{t}^2}({y} )  \subset S_{c_1{t}^2}^{Q_{z}}( {y})  \subset B_{C{t}^2}({y} ) \]
	Note that
	\[\begin{cases}
		\det  D^{2}u (x )  = 1      & \text{ in } S_{c_1{t}^2}^{Q_{z}}( {y}), \\
		|u-Q_{z}(x ) |\leq C|{t}|^{2+\alpha}                  & \text{ on } \partial S_{c_1{t}^2}^{Q_{z}}( {y} ).
	\end{cases} \]
	 The classical interior $C^{2,\alpha}$  theory for Monge-Amp\`ere equations then gives $ \left\|u\right\|_{C^{2,\alpha}(B_{c{t}^2}({y} ))}\leq C$, hence $u\in C^{2,{\alpha} }\left({\Omega} \cap B_{\rho}(0)\right)$ for some small universal $\rho$.
\end{proof}

\begin{Lemma} \label{lem:station perturb 1}
	Let $\epsilon>0$, suppose   $u\in C(\overline{S_1})$ satisfies
	\[\begin{cases}
		\left|\operatorname{det}D^2 {u}-1\right| \leq  \delta\epsilon    &  \text{ in } {S}_1\\
		\left|D_{n} {u} -a x_1\right| \leq  \delta\epsilon              &  \text{ on }  {G}_1 \\
		{u}=1                   &  \text{ on }  \partial {S}_1 \setminus {G}_1
	\end{cases}\]
 Assume that    for some $P(x)=\ell(x') +Q(x) \in \F_{a}$ with $B_{c/4}(0) \subset \left\{P \leq 1 \right\}$,  we have
	\[ \left\|g\right\|_{C^{1,\alpha}(B_{C}'(0))}+\left\|u-P\right\|_{L^{\infty}\left(S_1  \right)} \leq  \epsilon,\]
	then  $	|D\ell|\leq C\epsilon^{\frac{1}{2}}$.
	For any small constant $\mu>0$, there exists small positive $\delta_0$ and $\epsilon_0$, depending on $n$, $|a|$, $\mu$ such that whenever $\epsilon \leq \epsilon_0 $ and $\delta \leq \delta_0 $, there is a better approximation $P^0=l^0+ Q^0 \in \F_{a}$ of $u$ such that
	\[	|u-P^0| \leq C \epsilon  \mu^{\frac{3}{2}} \text{ in } S_{\mu}^{Q^0} \text{ and } \left\| D^2 Q^0 - D^2 Q \right\| \leq 2\epsilon.\]
\end{Lemma}
\begin{proof}
	  By applying Lemma \ref{lem:visc exist mix}, we  obtain a solution $w $ of
	 \[	\begin{cases}
	 	\det D^2 w =1   & \text{ in }{S_1^u(0)}, \\
	 	D_{n} {w } =ax_1             &  \text{ on }{G_\frac{7}{8}^u(0)},  \\
	 	u^- \leq w  \leq u^+ & \text{ in }{S_1^u(0)},
	 \end{cases} \]
 In fact, for $C$ sufficiently large, we can take \[u^-:=(1+C\delta\epsilon)(u+C^2\delta\epsilon x_n)-2C^3 \delta\epsilon\]
 and   
 \[u^+:=(1-C\delta\epsilon)(u-C^2\delta\epsilon x_n)+2C^3 \delta\epsilon\]
 are the subsolution and supersolution, respectively.
 
	 We can see that 
	 \[ \left\|u-w  \right\|_{L^{\infty}} \leq C\delta\epsilon,\]
	 and then
	 \[\left\|g\right\|_{C^{1,\alpha}(B_{C}'(0))}+\left\|w-P\right\|_{L^{\infty}\left(S_1  \right)} \leq  2\epsilon.\]
	 According to Theorem \ref{thm:stationary thm}, $w$ is uniformly $C^{2,\alpha}$ up to the Neumann boundary. The $C^1$ estimate for the Neumann problem of the uniformly elliptic equation implies   
	 	\[	\left\|g\right\|_{C^{1,\alpha}(B_{c}'(0))}+\left\|w-P\right\|_{L^{\infty}\left(S_{3/4}  \right)} + \left\|D_nw-D_nP\right\|_{L^{\infty}\left(S_{3/4} \right)}  \leq  C\epsilon.\]
	  Thus, by applying Lemma \ref{lem:stationary lem}, we obtain a better approximation of $w$, and therefore, a better approximation of $u$, provided that $\delta \leq \delta_0$  is sufficiently small.  
\end{proof}

\begin{proof}[\bf Proof of Theorem \ref{thm: c2alpha n=3}] 

  By Lemma \ref{lem:sta c0 perturb},  we may assume that $x_0=0$ and start with a suitable normalization of $(\tilde{u},\tilde{S}_1)$ at $h=h_0$, where $h_0$ is small enough. Then, $\tilde{u}$ solves
\[	\begin{cases}
	\det D^2 \tilde{u}=\tilde{f}       & \text{ in }\tilde{S}_1 , \\
	D_{n} \tilde{u} =\tilde{\phi^0 }                &  \text{ on }\tilde{G}_1 , \\
	\tilde{u}=1                   &  \text{ on }\partial  \tilde{S}_1 \setminus \tilde{G}_1 ,
\end{cases}\]
where
\[	\tilde{f}(x) =    \frac{(\det  \D)^2}{h^n}f(\T_h x),\quad \tilde{\phi^0 }(x') =\frac{d_n\phi^0  ( \T'_h x')}{h}.\]	
Here, we may assume that $\tilde{u}(0)=0$, $\nabla \tilde{u}(0)=0$, $ \tilde{u}\geq 0$,    $B_{c}^+(0)\cap \overline{ \tilde{S}_1} \subset  \tilde{S}_1 \subset B_{C}(0)$  and let  the defining function $ \tilde{g}$ of $G_1$ satisfies $\left\| \tilde {g}\right\|_{Lip(B_{c}'(0))}  \leq C$. 
We also assume that 
\[  \tilde{f}_{h_0}(0) =1,\quad D  \tilde{\phi}_{h_0}(0) =ae_1,\quad |a| \leq C,\] 
and there exists $P(x)=\ell(x') +Q(x) \in \F_{a}$ with  $B_{c/4}(0) \subset \left\{P \leq 1 \right\}$ satisfying
\[ \left\|\tilde{u}-P\right\|_{L^{\infty}\left(\tilde{S}_1  \right)} \leq  \epsilon_0.\]

  The proof is then the same as in Theorem \ref{thm:stationary thm}. By applying Lemma \ref{lem:station perturb 1}, we can still use the induction argument to prove that $h_k={h_0}\mu^k$, $\epsilon_k:= \epsilon_0 \mu^{^{\frac{k\alpha}{8}}}$ (and then for $\epsilon_k:= C\epsilon_0 \mu^{^{\frac{k\alpha}{2}}}$) the following.

 {\bf Claim.}\;   
 For $k\geq 0$, there exists a sequence of constants $\kappa_k \in (c,C)$ and transformations
 \[\M_k=\operatorname{diag}\left\{ \M_k',\M_{k,nn}\right\}, \quad \det  {\M_k} =1,\]
 such that at height $h_k$,  the  normalization solution $(u_{k}, \Omega_{k} )$ of $(u, \Omega)$ given by
 \[	u_{k}(x):= \frac{u(\T_k h_k^{\frac{\alpha}{2}} x )}{h_{k}} \text{ for } \in  \Omega_k:=  \T_k^{-1}  (\Omega) \text{ with } \T_k=\T_{h_0}\Pi_{i=0}^{k} \M_k.\]
 satisfies
 	\[\begin{cases}
 	\left|\operatorname{det}D^2 {u_k}-1\right| \leq  \delta\epsilon_k    &  \text{ in } {S}_1\\
 	\left|D_{n} {u_k} -a x_1\right| \leq  \delta\epsilon_k              &  \text{ on }  {G}_1 \\
 	{u_k}=1                   &  \text{ on }  \partial {S}_1 \setminus {G}_1
 \end{cases}\]
And we have
 \[	 \left\|g_k\right\|_{C^{1,\alpha}(B_{C}'(0))}+\left\|u_k-P_k\right\|_{L^{\infty}(S_1)} + \left\|D_nu_k-D_nQ_{a, \kappa_k}\right\|_{L^{\infty}(S_1)} \leq \epsilon_{k} \leq  \epsilon_0\]
 for some $P_k=\ell_k+Q_{a, \kappa_k} \in \F_{a}$, where $g_k$ is the defining function of the oblique boundary of $u_k$, and 
 \[	Q_{a,\kappa} (x):= \frac{1}{2}\left[\sqrt{\kappa^{n-2}+a^2}(x_1^2+x_n^2)+2ax_1x_n+\kappa^{-1}\sum_{i=2}^{n-1}x_i^2\right].\]
 Moreover, 
 \[\left\|{\M}_k-\I\right\|+|\kappa_k-\kappa_{k-1}|\leq C\epsilon_{k-1}  \text{ and } |D \ell_k| \leq C\epsilon_{k-1}^{\frac{1}{2}}. \]


  And the same proof gives the $C^{2,\alpha}$ estimates of the solution as that of Theorem \ref{thm:stationary thm}. When $n=2$, the $C^{2,\alpha}$ regularity is holds at every boundary point, and for points far away from the boundary, the maximum height of their sections inside the domain has a positive lower bound, giving a global $C^{2,\alpha}$ estimate.
\end{proof}
%
%

%
%
%

\vspace{8 pt}

%
%
%

\vskip 0.5cm

\section{Examples}\label{sec:example}

In this section, we will give some examples of the oblique derivative problem by introducing convex functions for which $\det D^2 u \approx 1$ in $B_{\rho}(0)$,  to illustrate the obstacles for the boundary Schauder regularity in the oblique derivative problem for Monge-Amp\`ere equations when the dimension $n\geq 3$.

\begin{Example}\label{exam:pogorelov function}
	Consider classical Pogorelov's function
	\[ u(x)=(1+x_n^2)|x'|^{2-\frac{2}{n}}  \text{ in } B_{\rho}(0) , \quad n\geq 3. \]
	Here, $u$ is only  $C^{1,1-2/n} \cap  C^{\infty}(B_{\rho}(0) \setminus \left\{|x'|=0\right\})$, the oblique boundary value $D_{\nu}u$ is  $C^{1,1-2/n}$ on boundary $\partial B_{\rho}(0)$. Moreover, we can take a smooth  oblique vector field  $\beta $ such that $\beta( x',x_n)= ( -x'x_n,  ( 1-\frac{1}{n})(1+x_n^2))$ around the points $\left\{ \pm \rho e_n\right\}$ such that $u$ solves \eqref{eq:oblique eq 1} with both $\Omega:= B_{\rho}(0)$,   $f$, $\beta$, $\phi$ being smooth.
\end{Example}

\begin{Example}\label{exam:Mooney singular}
Consider the merely Lipschitz function appearing in Caffarelli \cite{[C35]}
\[u\left(x \right)=\left|x^{\prime}\right|+\left(1+x_{n}^{2}\right) \left|x^{\prime}\right|^{n / 2} \text{ in } B_{\rho}(0) ,\quad n\geq 3. \]
Here, the Neumann boundary value $D_{\nu}u$ is Lipschitz on boundary $\partial B_{\rho}(0)$, and $\det D^2u$ is as  smooth as $ \left(1+\frac{n}{2} \left|x^{\prime}\right|^{n / 2-1} (1+x_n^2)\right)^{n-2}$. And we can always take an appropriate smooth oblique vector field  $\beta $ such that  $\phi= D_{\beta}u $ is as  smooth as $\left|x^{\prime}\right|^{n / 2}$.
\end{Example}

Finally, we recall a singular homogeneous function from our recent paper \cite[Example 4.3 and Remark 4.4]{[JT]}.	Suppose $n \geq 3$, for any small $\delta >0$ and  $a,b\in (1,\infty)  $ satisfying 
\[\frac{1}{a}=\frac{1}{2}+\frac{\delta}{n-2} \text{
	and } b= \frac{1}{1-\delta},\]
we let
\[	W_{a,b}(x_1,x'')=\begin{cases}  |x''|^{a}+|x''|^{a-\frac{2a}{b}}|x_1|^2,    & \text{ if } |x''|^a \geq |x_1|^b ,  \\
	\frac{2b+a-ab}{b}|x_1|^{b}+\frac{ab-a}{b}|x_1|^{b-\frac{2b}{a}}|x''|^2,& \text{ if }  |x_1|^b \geq |x''|^a  ,
\end{cases}\]
and	
\[	W(x_1,x'',x_n)= (1+x_n^2)W_{a,b}(x').\]

Let $\rho=\rho(\delta) >0$ be small. Then,  the function $ W $ is convex in $\left\{|x_n| \leq \rho \right\}$ and satisfies
\[	c(\delta)\leq   \det  D^2W \leq C(\delta), \quad D_nW(x',0)=0.\]
Moreover, given a constant $R>0$,  the solution $W_R^+$  to the Dirichlet problem
\[	\begin{cases}
	\det  D^2 W^+ =c(\delta)  & \text{ in } B_{R}'(0) \times [-\rho, \rho], \\
	W^+ =W  & \text{ on } \partial (B_{R}'(0) \times [-\rho, \rho])
\end{cases} \]
satisfies
\begin{equation}\label{eq8.1 compare W}
	W(x',0) \leq W^+(x',t) \leq  (1+\rho^2)W(x',0)
\end{equation}
Additionally, we will verify this in Appendix \ref{app:proof exam} the following Example \ref{exam:c11}, which shows that the assumption  \eqref{eq:qg ch1} in Theorem \ref{thm: c2alpha n=3} is optimal in some sense for local problems.

\begin{Example}\label{exam:c11}
	The function   $u(y_1,y'',y_n) =W^+(y_n,y'',y_1)$ satisfies $u \in C_{loc}^{1,a-1}(B_{\frac{\rho}{2}}'(0))$, and  is a solution to
	\begin{equation}\label{eq8.2 local problem example}
			\det  D^2 u =c \text{ in } B_{\frac{\rho}{2}}^+(0), \
		D_n u =0  \text{ on } B_{\frac{\rho}{2}}'(0).
	\end{equation}
But $u\notin C^{1,\delta+\epsilon}$ for any $\epsilon >0$.
\end{Example}


\appendix  

\section{Proof of Lemma \ref{lem:stationary lem}}\label{app:proof b}

\begin{proof}
	The estimate of $|D\ell|$ is obtained based on convexity, which we omit here. Contrary to our statement \eqref{eq7.4 better approximation 1}, there exists a sequence of solutions $u_k$ to 
	\[	\begin{cases}
		\operatorname{det}D^2 {u_k}=1     &  \text{ in } {S}_1^{u_k},\\
		D_{n} {u_k} =a_k x_1               &  \text{ on }  {G}_1^{u_k} ,\\
		{u_k}=1                   &  \text{ on }  \partial {S}_1^{u_k} \setminus {G}_1^{u_k},
	\end{cases}\] 
with good approximations  $P_k=\ell_k+Q_k $ satisfying
	\[	\epsilon_{k}:=\left\|g_k\right\|_{C^{1,\alpha}(B_{C}'(0))}+\left\|u_k-P_k\right\|_{L^{\infty}(S_1^{u_k})}+ \left\|D_nu_k-D_nQ_k\right\|_{L^{\infty}(S_1^{u_k})}   \to 0 ,\]
	 but \eqref{eq7.4 better approximation 1} fails for $u_k$ and all polynomials in $\F_{a_k}$.

	According to the assumption $B_{c/4}(0) \subset \left\{P_k \leq 1 \right\}$, we always have $c\I \leq D^2 Q_k \leq C\I$. By considering a subsequence and using affine transformations, without loss of generality, we may assume that for all $k$, 
	\[	Q_{k} (x):= \frac{1}{2}\left[\sqrt{\kappa_{k}^{n-2}+{a_k}^2}\left(x_1^2+x_n^2\right)+2 {a_k}x_1x_n+ \kappa_{k}^{-1}\sum_{i=2}^{n-1}x_i^2\right].\]
		We assume that $\lim_{k\to\infty} \kappa_k=1$ and $\lim_{k\to\infty} a_k=a$, and let
	\[	Q_{a} (x):= \frac{1}{2}\left[\sqrt{1+a^2}\left(x_1^2+x_n^2\right)+2 ax_1x_n+ \sum_{i=2}^{n-1}x_i^2\right].\]
	 In the following discussion, we omit all subscripts $u_k$ for the domains and boundaries, and we consider the functions 
	\[ w_k = \frac{u_k-\ell_k-Q_{k}}{\epsilon_{k}}. \]

	\textbf{Step 1}. As $k \to \infty$, $w_k$ converges locally uniformly to a function $w$ satisfying
	\[	\L[w]:= \sqrt{1+a^2}(w_{11}+w_{nn})+ a w_{1n} + \sum_{i=2}^{n-1} w_{ii}=0 \text{ in }  B_c^+(0).\]
	
	We consider convergence on the compact subset
	\[E=\left\{ x\in \R^{n}|\, x_n \geq C\epsilon_k^{\frac{1}{2}} \right\} \cap  S_{1}(0) \cap S_{1/2}^{Q_k}(0).\]
	For every point $y \in E$ satisfying $\operatorname{dist}(y , \partial G_1) >> C\epsilon_k^{\frac{1}{2}}$, we choose a constant $\rho=c\operatorname{dist}(y , \partial G_1)$
	such that the ball $B_{C\rho}(y ) \subset E $. Then we have
	\[B_{c\rho}(y) \subset  S_{\rho^2}^{Q_k}(y) \subset B_{C\rho}(y ).\]
	Recall that  $ |{{u}_k}-Q|\leq  \epsilon_k$. Since $\epsilon_k \leq c\rho^2$, we have $B_{2c\rho}(y ) \subset S_{\rho^2/2}^{{u}_k}(y) \subset B_{C\rho}(y )$. The standard estimate of the Monge-Amp\`ere equation implies that 
	\[   \left\|{{u}_k}\right\|_{C^{m}(B_{c\rho}(y))}  \leq C_m  \rho^{2-m}  \text{ and } c\I \leq  D^2{{u}_k}\leq C\I \text{ in } B_{c\rho}(y) \text{ for } m\geq 0.\]
	Note that
	\begin{equation} \label{eq7.6 diff 8}
		0 =  \det  D^{2}{u}_k-\det  D^{2} Q_k=Tr  ( AD^2({{u}_k}-Q_k ))= \epsilon_k Tr  ( AD^2 w_k),
	\end{equation}
	where
	\[  A:=[A_{ij}]_{n\times n} =\int_0^1 cof ( ( 1-t )D^2Q_k+tD^2{{u}_k} ) dt \]
	and $cof \M$ denotes the cofactor matrix of a matrix $\M$. The operator $\L_A[\cdot]:=A_{ij}D_{ij}(\cdot)$ is uniform elliptic with smooth coefficients. From \eqref{eq7.6 diff 8}, we obtain
	\begin{equation} \label{eq7.7 nter esti 8}
		\left\|{w}_k\right\|_{C^{m}(B_{c\rho}(y))}\leq C\rho^{-m},
	\end{equation}
	which implies
	\[		| D^{2}{{u}_k}- D^{2} Q_k| \leq C\epsilon_k \rho^{-2} \text{ and } | A_{ij}- D_{ij} Q_k|\leq C\epsilon_k \rho^{-2}.\]
	Letting $ \epsilon_k \to 0$,  we see that $({w}_k,D_n {w}_k)$ converges locally uniformly in $B_{c\rho}(y)$ to $(w,D_nw)$ with $\L w=0$.

	\textbf{Step 2}. By giving an uniform control of $(w_k,D_nw_k)$ near $\left\{ x_n =0\right\}$, we have that
	\begin{equation}\label{eq7.8 linear limit station}
		\L w= 0 \text{ in }  B_c^+(0),\quad w(0)=0 ,\quad  D_nw= 0\text{ on } \left\{ x_n=0\right\}.
	\end{equation}
	
	We claim that
	\begin{equation} \label{eq7.9 linear limit station}
		|D_n w_k (x)| \lesssim x_n+\epsilon_k \text{ in }  S_{\frac{c}{2}}(0),
	\end{equation}
	and
	\begin{equation} \label{eq7.10 continuous universal esti 8}
		\omega_{w_k}(B_r(x))  \lesssim \max\left\{r^{\frac{2}{3}},\epsilon_k^{\frac{1}{2}}\right\} \text{ for } x \in G_{\frac{c}{2}}.
	\end{equation}
	
	Fixing  $k$, since we can construct a smooth approximation of $u_k$ using the solution of Dirichlet problems, as we did in Lemma \ref{lem:c11 normal 2}, without breaking the structure of the problem, we may assume that $w_k $ is smooth.Then, ${\zeta} =D_n w_k$ satisfies
	\[\begin{cases}
		U_k^{ij} D_{ij}{\zeta}=0        & \text{ in } S_1(0), \\
		|{\zeta}|\leq C & \text{ in } S_1(0),  \\
		|{\zeta}|=0   & \text{ on }  {G_1(0)}.
	\end{cases}\]
	where $U_k^{ij}$ denote the cofactor matrix of $D^2 u_k$. Let $C_1$ and $C_2 $ be sufficiently large, the function
	\[\upsilon(x) =C_1\left[ u_k(x)-\frac{n}{2}x_nD_n u_k+C_2\left(x_n+  \epsilon_{k}\right)  \right] .\]
	satisfies
	\[\begin{cases}
		U_k^{ij} D_{ij}\upsilon =0        & \text{ in } S_1(0), \\
		\upsilon\geq C    & \text{ on } \partial S_1(0) \setminus G_1(0)  , \\
		\upsilon  \geq 0             & \text{ on }  {G_1(0)}.
	\end{cases}\]
	Hence, $   \upsilon$ and $- \upsilon$  are the upper and lower barriers of ${w}_k$ at $0$, respectively. Therefore,
	\[ |D_n {w_k}(te_n)| \lesssim t+\epsilon_{k}.\]
	By considering the function $ u_{k,x_0}= [ u_k (x)  - u_k (x_0)- \nabla  u_k (x_0) \cdot (x-x_0)  ]$ and transformation $\A x=(x-x_0)-D'g_k(x_0) \cdot x' e_n$, a similar discussion applies to every point on $G_{\frac{c}{2}}$. Therefore, we have
	\[	|D_n {w_k}(x)|\lesssim |\A x \cdot e_n|+\epsilon_{k} \lesssim x_n+\epsilon_{k}, \]
	and \eqref{eq7.9 linear limit station} is proved.
	
	From \eqref{eq7.9 linear limit station} that, we conclude that
	\[	|w_k(x', g_k(x'))-w_k(x', x_n)| \lesssim x_n^2+\epsilon_{k} x_n.\]
	To prove \eqref{eq7.10 continuous universal esti 8}, we estimate $| w_k (p)- w_k(q)|$ for points
	$p=\left(y_1,g_k(y_1)\right)$ and $q=\left(q_1,g_k(q_1)\right)$ on $G_{\frac{c}{2}}(0)$.  Let \[r=|p'-q'|  \text{ and } \rho=C\max\left\{r^{\frac{1}{3}},\epsilon_k^{\frac{1}{2}}\right\}\]
	and take points $y=(p',\rho )$ and $z=(q',\rho )$. Since $\rho \geq C\epsilon_k^{\frac{1}{2}}$, \eqref{eq7.7 nter esti 8} implies that
	\[ |{w_k}(y)-{w_k}(z)| \lesssim r\rho^{-1},\]
	and then
	\[ |{w_k}(p)-{w_k}(q)| \lesssim \rho^2 +\epsilon_{k} \rho +r\rho^{-1} \lesssim \max\left\{r^{\frac{2}{3}},   \epsilon_k^{\frac{1}{2}}\right\}.\]
	This gives \eqref{eq7.10 continuous universal esti 8}.
	
	\textbf{Step 3}. Now, $w$ is a continuous viscosity solution of  \eqref{eq7.8 linear limit station}, from \eqref{eq7.8 linear limit station} and the standard theory for  elliptic partial differential equations, we have  $w \in C_{loc}^3  $. Note that $w(0)=D_nw(0)=D'D_{n}w(0)=0$, There exists a quadratic function
	\[R(x)= \sum_{1\leq i,j \leq n-1} a_{ij}x_ix_j+\frac{a_{nn}}{2}x_n^2+\sum_{i=1}^{n-1} b_ix_i,\]
	such that
	\[	|w(x) - R(x)|\lesssim |x|^3  \text{ and }	 |D_n w(x) -a_{nn}x_n | \lesssim |x|^2,\]
	for which $ a_{ij},a_{nn},b_{i}$ are bounded and satisfy  $\sqrt{1+a^2}(a_{11}+a_{nn})+ \sum_{i=2}^{n-1}a_{ii}= 0$.
	As $\epsilon_k\to 0$,  recalling that $w$ is the limit of $w_k$ in   Step 2, we have
	\[	|{w_k}(x) - R(x)| \leq \sigma+ C|x|^3  \text{ and }	 |D_n {w_k}(x) -a_{nn}x_n | \leq \sigma+ C|x|^2,\]
	where $\sigma =\sigma( \epsilon_k ) \to 0  $. Suppose 
	\[\left|a_k-a\right|+ \left|\kappa_k-1\right|+\sigma^{\frac{2}{3}} +\epsilon_k ^{\frac{1}{3}} \leq \mu,\]
	  then
	\begin{equation} \label{eq7.11 step 3 1}
		\left|u_k- \left(Q_k+\epsilon_k  R\right)\right|   \leq C \epsilon_k\mu^{\frac{3}{2}} \text{ in }  S_{\mu}^{Q_k},
	\end{equation}
	and
	\begin{equation} \label{eq7.12 step 3 2}
		\left|D_n u_k -\left(D_nQ_k+ \epsilon_kD_nR\right)\right|  \leq C \epsilon_k\mu \text{ in }  S_{\mu}^{Q_k}.
	\end{equation}
	
	Fixed $k$, observing that  $\det D^2 u_k  =\det \left(D^2 Q_k+ \epsilon_kD^2 w_k\right)=1$,  we see that the function $k(t)=\det \left(D^{2} Q_k+\epsilon_kD^2R+t \I\right)$ satisfies $|k(0)-1| \leq  \epsilon_k\left( \left|a_k-a\right|+ \left|\kappa_k-1\right|\right)+O(\epsilon_k^2)\lesssim\epsilon_k \mu $ and $ k'(t) \approx 1$. The equation $k(t)=1$ has a solution $t_0$ satisfying
	\begin{equation}\label{eq7.13 step 3 3}
		|t_0| \lesssim \epsilon_k \mu.
	\end{equation}
	Take
	\[	\P^0 (x)= Q_k(x)+ \epsilon_kR(x) +\frac{t_0}{2}|x|^2,\]
	we have $\det  D^{2}P^0=1,\quad D_nP^0= ax_1$ on $ \left\{ x_n=0\right\}$. Thus, $P^0 \in \F_a$.
	Combining \eqref{eq7.11 step 3 1},  \eqref{eq7.12 step 3 2} and  \eqref{eq7.13 step 3 3}, we obtain
	\[ |{u_k }(x)-P^0(x)| \leq  Ct_0 \mu +C \epsilon_k\mu^{\frac{3}{2}} \leq   C \epsilon_k\mu^{\frac{3}{2}}  \text{ in }  S_{\mu}, \]
	which together with the result from Step 1 implies
	\[ |u_k(x)-P^0(x)| \leq    C \epsilon_k\mu^{\frac{3}{2}} +C\delta \epsilon\leq C \epsilon_k\mu^{\frac{3}{2}}  \text{ in }  S_{\mu}.\]
	This contradicts the definition of $u_k$ if we replace $\mu$ with $\mu/2$ and note that  $S_{\mu/2}(0) \cup S_{\mu/2}^{P_0}(0) \subset S_{\mu}$. Therefore, we have proved \eqref{eq7.4 better approximation 1}, and \eqref{eq7.5 better approximation Q} follows from calculation.
\end{proof}

\section{Verification of Example \ref{exam:c11}}\label{app:proof exam}

Similar to Lemma \ref{lem:lip on modules}, we can prove
\begin{Lemma}\label{lem2.8 lip bounded lemma 2}
	Write $x\in \R^n$ as $x=(x^1,\cdots,x^k)$, $x^j \in \R^{a_j}$, $\sum_{i=1}^{k} a_i=n$.
	Let $u\in C( B_r(0) )$ be a convex function, satisfying
	\[u(x) \leq u(0)+ \sum_{i=1}^{k} \sigma_i(|x^i|) \text{ in }\; B_r(0).\]
	Then
	\[		|\nabla_{x^j} u(0)| \leq C \inf_{ 0 \leq t \leq cr}\frac{2\sigma_i(Ct)}{t},\quad  i=1,2,\cdots, k .\]
	In particular, locally bounded convex functions are locally Lipschitz.
\end{Lemma}

Before verifying Example \ref{exam:c11}, we  introduce some properties of the singular function $W_{a,b}(x')$ on the plane $\R^{n-1}$. Denote
\[ v(x') =W_{a,b}(x') ,\quad \breve v(x')= (1+\rho^2)W_{a,b}(x').\]
Given $t>0$, consider the sections
\[ F_t =S_t^{ v}(0) ,  \breve F_t =S_t^{\breve v}(0)\]
and the diagonal transformation
\[ \D_t =\operatorname{diag}\left\{ t^{1/b}, t^{1/a}\I'' \right\}.\]

For simplicity, we only consider the case $t=1$ in the following lemmas since  $v$ is a homogeneous function and these lemmas are invariant under the normalization
\[ \tilde{v}(x):= \frac{v(\D_tx)}{t} =v(x).\]

\begin{Lemma}\label{l10.5}
	Given a linear function $ L(x')$ satisfying $L(x') \leq \breve v$.  Denote the section
	$S_L=\left\{ x'|\,  v\leq L(x') \right\}$.
	There exists universal constants $c,C$ such that if $S_L \cap \partial F_t \neq \emptyset $, then
	$ S_L \subset  \breve F_{Ct} \setminus  F_{ct}$.
\end{Lemma}
\begin{proof}
	
	Suppose that
	$	p \in  S_L \cap \partial F_1 \neq \emptyset $.
	Then $|p | \leq C$, the upper barrier relation $ L(x') \leq \breve v$ and $L(x') \geq 0$ implies that $\left\|D L\right\|= \left\|D L(p)\right\| \leq C$(see Lemma \ref{lem2.8 lip bounded lemma 2})). While the function $v$ is super-linearity at infinite, thus
	$S_L \subset B_{C}'(0) \subset  \breve F_{C_1}$
	for some universal constant $C_1$. This also implies the opposite relation when $c_1C_1$ is small enough.  Otherwise, if $S_L \cap  F_{c_1} \neq \emptyset$, then $S_L \subset  \breve F_{c_1C_1} \subset F_{\frac{1}{2}}$, which contradicts the assumption $S_L \cap \partial F_1 \neq \emptyset $.
\end{proof}

\begin{Lemma}\label{l10.6}
	Given  points  $p'$ and  denote $t = v(p')$. There exists universal constants $c, C$ such that  if $ s\in [0,c]$ then
	\[	c s^{\frac{1}{2}}\D_t B_1'(0)  \subset S_{st}^{\breve{v}}(p')-p'  \subset  C s^{\frac{1}{2}}\D_t B_1'(0) \]
	and thus
	\[\operatorname{Vol}_{n-1}S_{st}^{\breve{v}}(p') \geq cs^{\frac{n-1}{2}}t^{\frac{n}{2}}.\]
\end{Lemma}
\begin{proof}
	 We may assume   $t=1$. The strict convexity of $v$ implies that for some  small universal constant
 $c$,
	$$ S_c(p') \cap B_c'(0) = \emptyset\;  \text{ and}\;
S_c(0) \subset B_C'(0).$$
	Thus,
	$c \leq   D^2 v \leq C \text{ on } S_c(0)$
	and the lemma follows.
\end{proof}

\begin{Lemma}\label{l10.7}
	Suppose that function $ L(x')$ is a linear function such that $L(x') \leq \breve v$.  There exists universal constants $C$ such that if the section $S_L=\left\{ x'|\,  v\leq L(x') \right\} $ satisfies $S_L \cap \partial F_t \neq \emptyset $, then $\left\| t^{-1}\D_t \cdot D L \right\| \leq C$, and
	\begin{equation}\label{eq9.1}
		\breve{v}( x') -L(x') \leq C(v(x'-q')+t),\quad  \forall  q \in \breve F_{Ct} \setminus  F_{ct}.
	\end{equation}
\end{Lemma}
\begin{proof}
	Suppose that $t=1$. Then $\left\|\D L\right\| \leq C$ and $p' \in B_C'(0)$. Note that $v$ is a homogeneous function which is super-linearity at infinite. Thus, for $C_1$ large enough,
	\[C|x'| \leq  C_1v(x'-q') +C_1,\]
	and
	\[  \breve v(x') \leq C_1v(x'-q') +C_1.\]
	These two inequalities imply our lemma.
\end{proof}

 Now we are in the position to {\sl verify Example \ref{exam:c11}}, which will be completed by three Steps.

\textbf{Step 1}.	Denote $w(x)=W_R^+(x)$ and $E=B_{R}'(0) \times (-\rho, \rho)$, we claim that
\[w\in  C^{1}(E)\cap C_{loc}^{2}(E\setminus\left\{te_n\right\}) ,\]
and
\[  D_1 w(0,x'',x_n)=0 \text{ and }  D_n w(x_1,x'',0)=0.\]

Consider the set $\Gamma$  formed by the intersection of the image of $w$ and the lines $ l $.  Let   $x_0$ be a  endpoint of $l$.  The classical interior strict convexity lemma states that the endpoints of $\Gamma$ is not inside $E$, therefore $x_0$ is on the boundary. Assume by way of contradiction that $x_0'\neq 0$,  then $w$ is $C^{2}$ at  point $x_0$ for some $\epsilon < \frac{2}{n}$, the Pogorelov strict convexity lemma implies that $x_0$ cannot be the extreme point $\Gamma$. Therefore, $\Gamma$ is contained in the axis $e_n$. Thus, $w\in C^{2}(E \setminus \left\{ |x'|=0 \right\})$. Note that (\ref{lem2.8 lip bounded lemma 2}) and \eqref{eq8.1 compare W} imply $w$ is pointwise $C^{1,b-1}$ on axis $e_n$, the claim is completed.

\textbf{Step 2}. Given point $q =(q',q_n) \in E \cap B_{\frac{\rho}{2}}(0)$. Suppose that $q' \neq 0$. Denote $t= v(q')$. Let $S_{h_q}^{{w}}(q)=\left\{ x \in E |\, w(x) \leq w(q)+ \nabla w(q )\cdot (x-q) +h_q\right\} $ be the maximum section  with base point $p$, and $L(x)= w(x) \leq w(q)+ \nabla w(q )\cdot (x-q) +h_q$. We aim to prove that $h_q \approx t$, where $t = v(q') \approx w(q)$.

We claim $h_q \leq Ct$ by proving that
\begin{equation}\label{eq9.2}
	\P S_{h_q}^{{w}}(q) \subset  \breve F_{Ct} \setminus  F_{ct}.
\end{equation}
For each $r \in [-\rho,\rho]$, consider the function
\[ v_r(x') =w(x',r) \text{ and }  L_r(x')=L(x',r),\]
and the sections
\[H_r = \left\{ x' |\,v_r(x') \leq L_r(x) \right\},\]
\[G_r = \left\{ x' |\,v(x') \leq L_r(x) \right\}.\]
By \eqref{eq8.1 compare W}
\[H_r \subset G_r .\]
Since $q$ is near axis $e_n$, the maximum section touches $\partial E$ at point $p \in \left\{ |x_n|=\rho\right\}$.  For simplicity, we assume that $q_n \geq 0$, then
\[ D_n w (q) \geq D_nw(q',0) =0.\]
Therefore, for any constants $r_2 \leq r_1 \leq \rho$,
\[  L_{r_2}(x') \leq  L_{r_1} (x') \leq \breve{v}(x'),\]
and
\[ G_{r_2} \subset G_{r_1}  .\]
Thus, we can assume that $p_n =\rho$. Then
\[   q' \subset G_{q_n}\subset G_{p_n}.\]
Recall Lemma \ref{l10.5}, we obtain
\[     G_r \subset \breve F_{Ct} \setminus  F_{ct}.\]
This is \eqref{eq9.2}.

Next, we show that $h_q \geq ct$.
Let $s \geq 0$. By Lemma \ref{lem:sec vol lbd} ,
\[ \left| S_{h_q+st}^{{w}}(q)\right| \leq C(h_q+st)^{\frac{n}{2}}.\]
Recall \eqref{eq9.2},
\[ u(p)\geq ct.\]
Note that $L_{\rho}$ is the support function of $\breve{v}$ at point $p$. Lemma \ref{lem2.5 balance lemma} and Lemma  \ref{l10.6}  mean
\[ \left| S_{h_q+st}^{{w}}(q) \right|\geq c\rho \operatorname{Vol}_{n-1} \left\{ S_{h_q+st}^{{w}}(q) \cap \left\{x_n =\rho \right\}  \right\} \geq c\rho s^{\frac{n-1}{2}}t^{\frac{n}{2}}.\]
Thus,
\[C(h_q+st)^{\frac{n}{2}} \geq  c\rho s^{\frac{n-1}{2}}t^{\frac{n}{2}}. \]
i.e.
\[\frac{h_q}{t} \geq \sup_s \left\{c(\rho)s^{\frac{n-1}{n}} -s\right\} \geq c. \]

\textbf{Step 3}. We prove that w satisfies
\begin{equation}\label{eq9.3}
	w(x)-w(q)-\nabla w(x)\cdot (x-q)  \leq C(|x_1|^b+|x''|^a +x_n^2)
\end{equation}
for any $q \in B_{\frac{\rho}{2}}(0)$.

The case that $q $ is on axis $e_n$ follows from \eqref{eq8.1 compare W}. We then assume that $q'\neq 0$.

If $x\in S_{h_q}^{{w}}(q) $. Note that $S_{h_q}^{{w}}(q)$ is interior section, the first equation of problem \eqref{eq8.2 local problem example} means that
\[ \left|  S_{h_q}^{{w}}(q ) \right| \approx h_q^{\frac{n}{2}} \approx t^{\frac{n}{2}},\]
and
\[    \left| \breve F_{C^2t} \setminus  F_{c^2t} \times [-\rho,\rho] \right| \leq Ct^{\frac{n}{2}}.\]
Note that $S_{h_q}^{{w}}(q)$ is balanced about $q$. Therefore, according to Lemma \ref{lem2.5 balance lemma},
\[S_{s_0t}^{{w}}(q)-q \approx \D_t B_1'(0) \times [-\rho,\rho].\]
for $s_0 =\frac{h_q}{t} \approx c$. Then the classical interior $C^{1,1}$ regularity results implies  if $s \in [0,s_0]$, then
\begin{equation}\label{eq9.4}
	c s^{\frac{1}{2}}\D_t B_1'(0) \times [-cs^{\frac{1}{2}},cs^{\frac{1}{2}}] \subset S_{st}^{{w}}(q)-q  \subset  C s^{\frac{1}{2}}\D_t B_1'(0) \times [-Cs^{\frac{1}{2}},Cs^{\frac{1}{2}}].
\end{equation}
And \eqref{eq9.4} implies that $\eqref{eq9.3}$.

If $x\notin S_{h_q}^{{w}}(q) $. Then
\[ v(x'-q') \geq ct \text{ outside } S_{h_q}^{{w}}(q).\]
Recall \eqref{eq9.2}, we have $\left\|D_n L\right\| \leq Ct$. By \eqref{eq9.1},
\begin{align*}
	w(x)-w(q)-\nabla w(x)\cdot (x-q)  & \leq 	\breve{v}( x') -L_{q_n}(x')  +Ct  \\
	& \leq C(v(x'-q')+t) \\
	& \leq Cv(x'-q').
\end{align*}
This is \eqref{eq9.3}.

In particular, by virtue of  Lemma \ref{lem2.8 lip bounded lemma 2} and  \eqref{eq9.3}, we see   that $w$ is $C^{1,1-a}$ on the plane $x_1=0$, thus we finish the proof.


\vskip 1cm


\end{document}